\documentclass[11pt,twoside]{article}
\usepackage{mathrsfs}
\usepackage{amssymb}
\usepackage{amsmath}
\usepackage[all]{xy}
\usepackage{amsthm}
\usepackage{url}
\numberwithin{equation}{section}

\newtheorem{theorem*}{Theorem}

\theoremstyle{definition}

\theoremstyle{remark}

\usepackage{paralist}

\numberwithin{equation}{section} \theoremstyle{plain}
\newtheorem*{thm*}{Main Theorem}
\newtheorem{thm}{Theorem}[section]
\newtheorem{cor}[thm]{Corollary}
\newtheorem*{cor*}{Corollary}
\newtheorem{lem}[thm]{Lemma}
\newtheorem*{lem*}{Lemma}
\newtheorem{prop}[thm]{Proposition}
\newtheorem*{prop*}{Proposition}
\newtheorem{rem}[thm]{Remark}
\newtheorem{ques}[thm]{Question}

\newtheorem*{rem*}{Remark}
\newtheorem{exa}[thm]{Example}
\newtheorem*{exa*}{Example}
\newtheorem{df}[thm]{Definition}
\newtheorem*{df*}{Definition}

\newtheorem*{conj*}{Conjecture}

\newtheorem*{Fa*}{Fact}

\newtheorem*{Qu*}{Question}
\newtheorem*{ack*}{ACKNOWLEDGEMENTS}

\newcommand{\pf}{\noindent\begin {proof}}
\newcommand{\epf}{\end{proof}}

\newcommand{\Ext}{\mbox{\rm Ext}}
\newcommand{\Hom}{\mbox{\rm Hom}}
\newcommand{\Tor}{\mbox{\rm Tor}}

\def\Im{\mathop{\rm Im}\nolimits}

\def\Mod{\mathop{\rm Mod}\nolimits}

\def\fd{\mathop{\rm fd}\nolimits}
\def\id{\mathop{\rm id}\nolimits}
\def\pd{\mathop{\rm pd}\nolimits}

\def\min{\mathop{\rm min}\nolimits}
\def\sup{\mathop{\rm sup}\nolimits}
\def\inf{\mathop{\rm inf}\nolimits}
\def\add{\mathop{\rm add}\nolimits}
\def\Add{\mathop{\rm Add}\nolimits}
\def\Prod{\mathop{\rm Prod}\nolimits}

\def\gldim{\mathop{\rm gl.dim}\nolimits}

\def\Gpd{\mathop{\rm G\text{-}pd}\nolimits}
\def\SGFpd{\mathop{\rm SGF\text{-}pd}\nolimits}
\def\Gid{\mathop{\rm G\text{-}id}\nolimits}
\def\GFPid{\mathop{\rm GFP\text{-}id}\nolimits}
\def\GFPCid{\mathop{\rm GFP_{\it C}\text{-}id}\nolimits}
\def\FPid{\mathop{\rm FP\text{-}id}\nolimits}

\def\GCpd{\mathop{\rm G_{\it C}\text{-}pd}\nolimits}
\def\SGFCpd{\mathop{\rm SGF_{\it C}\text{-}pd}\nolimits}

\def\GCid{\mathop{\rm G_{\it C}\text{-}id}\nolimits}

\def\Hom{\mathop{\rm Hom}\nolimits}
\def\Ext{\mathop{\rm Ext}\nolimits}
\def\sup{\mathop{\rm sup}\nolimits}

\def\lim{\mathop{\underrightarrow{\rm lim}}\nolimits}

\def\res{\mathop{\rm res}\nolimits}

\def\cores{\mathop{\rm cores}\nolimits}

\def\End{\mathop{\rm End}\nolimits}

\newcommand{\A}{\mathscr{A}}
\newcommand{\C}{\mathscr{C}}
\newcommand{\D}{\mathscr{D}}
\newcommand{\U}{\mathscr{U}}
\newcommand{\V}{\mathscr{V}}
\newcommand{\X}{\mathscr{X}}
\newcommand{\Y}{\mathscr{Y}}
\newcommand{\E}{\mathscr{E}}
\newcommand{\F}{\mathscr{F}}

\begin{document}
\begin{center}
{\large \bf Generalized Gorenstein Categories}
\footnote{The research was partially supported by NSFC (Grant Nos. 12371038, 12171207).}

\vspace{0.5cm}
Zhaoyong Huang\\
{\footnotesize\it School of Mathematics, Nanjing University, Nanjing 210093, Jiangsu Province, P.R. China}\\
{\footnotesize\it E-mail: huangzy@nju.edu.cn}

\vspace{0.5cm}

{\footnotesize\it Dedicated to Professor Claus Michael Ringel on the occasion of his eightieth birthday}
\end{center}


\bigskip
\centerline { \bf  Abstract}
\bigskip
\leftskip10truemm \rightskip10truemm \noindent

Let $\mathscr{A}$ be an abelian category and let $\mathscr{C}$ and $\mathscr{D}$ be additive subcategories of $\mathscr{A}$.
As a generalization of Gorenstein categories, we introduce one-sided $n$-$(\C,\D)$-Gorenstein categories with $n\geq 0$. 
Under certain conditions, we give some equivalent characterizations of one-sided $n$-$(\C,\D)$-Gorenstein categories
in term of the finiteness of projective and injective dimensions relative to one-sided Gorenstein subcategories, which 
induce some new equivalent characterizations of Gorenstein categories. Then we apply these results to categories of interest. 
In particular, a necessary condition is obtained for the validity of the Wakamatsu tilting conjecture.
\vbox to 0.3cm{}
\vbox to 0.2cm{}
\leftskip0truemm \rightskip0truemm

\section { \bf Introduction}

Let $\mathscr{A}$ be an abelian category and $\mathscr{C}$ an additive subcategory of $\mathscr{A}$. As a common generalization
of Gorenstein projective and Gorenstein injective modules, Sather-Wagstaff, Sharif and White \cite{SSW} introduced the
Gorenstein subcategory $\mathcal{G}(\mathscr{C})$ of $\mathscr{A}$ relative to $\mathscr{C}$. 
From the definition of the Gorenstein subcategory $\mathcal{G}(\mathscr{C})$,
it is known that $\mathscr{C}$ should be simultaneously a generator and a cogenerator for $\mathcal{G}(\mathscr{C})$
and both functors $\Hom_\A(\C,-)$ and $\Hom_\A(-,\C)$ should possess certain exactness.
It leads to some limitations. To overcome such limitations, Song, Zhao and Huang \cite{SZH} introduced 
one-sided Gorenstein subcategories by modifying the definition of Gorenstein subcategories, and then 
Gao and Wu \cite{GW} generalized them to a more general setting. On the other hand, Beligiannis and Reiten \cite{BR}
introduced Gorenstein categories by categorizing Gorenstein rings, and compare the relevant results about Gorenstein rings 
to Gorenstein categories. It is a natural topic to study the relation between Gorenstein categories and 
(one-sided) Gorenstein subcategories, we will study it in a more general setting.

In homological theory, homological dimensions are important and fundamental invariants, 
which play a crucial role in studying the structures of modules and rings. 
The properties of modules and rings have been studied when certain homological dimensions 
relative to some special one-sided Gorenstein subcategories of a given module or all modules 
are finite, see \cite{EJ,GW,H,Hu3,HS,SZH}, and so on.
In this paper, we will introduce generalized Gorenstein categories and give them some equivalent characterizations 
in term of the finiteness of certain homological dimensions relative to one-sided Gorenstein subcategories,
which provide a unified framework for some related results.

The paper is organized as follows.
In Section 2, we give some notions and preliminary results which will be used in the sequel.

Let $\A$ be an abelian category and let $\mathscr{C}$ and $\mathscr{D}$ be additive subcategories of $\mathscr{A}$.
In Section 3, we introduce left or right $n$-$(\C,\D)$-Gorenstein categories with $n\geq 0$.
Assume that $\A$ has enough projective and injective objects and satisfies
the AB4 and AB4$^{*}$ conditions $($that is,
$\mathscr{A}$ has arbitrary direct sums and products which are exact$)$ and
$\mathscr{C}$ is closed under direct summands. We prove the following result.

\begin{thm} \label{thm-1.1} {\rm (Parts of Theorems \ref{thm-3.11} and \ref{thm-3.14})}
It holds that
\begin{enumerate}
\item[{\rm (1)}] 
Let $r\mathcal{G}(\mathscr{C},\mathscr{D})$ be the right Gorenstein subcategory of $\A$
relative to $\C$ and $\D$, which contains $\mathcal{P}(\A)$ consisting of all projective objects in $\A$. 
If $\mathscr{C}$ is closed under kernels of epimorphisms and 
$\C$ is contained in the intersection of $\mathscr{D}$ and its left orthogonal category,
then the following statements are equivalent.
\begin{enumerate}
\item[{\rm (1.1)}] $\A$ is a right $n$-$(\C,\D)$-Gorenstein category. 
\item[{\rm (1.2)}] The $r\mathcal{G}(\C,\D)$-projective dimension of any object in $\A$ is at most $n$.
\item[{\rm (1.3)}] The three categories respectively consisting of objects with $\D$-projective, injective
and $\C$-projective dimensions at most $n$ coincide.
\end{enumerate}
\item[{\rm (2)}] 
Let $l\mathcal{G}(\mathscr{C},\mathscr{D})$ be the left Gorenstein subcategory of $\A$
relative to $\C$ and $\D$, which contains $\mathcal{I}(\A)$ consisting of all injective objects in $\A$. 
If $\mathscr{C}$ is closed under cokernels of monomorphisms and
$\mathscr{C}$ is contained in the intersection of $\mathscr{D}$ and its right orthogonal category,
then the following statements are equivalent.
\begin{enumerate}
\item[{\rm (2.1)}] $\A$ is a left $n$-$(\C,\D)$-Gorenstein category. 
\item[{\rm (2.2)}] The $l\mathcal{G}(\C,\D)$-injective dimension of any object in $\A$ is at most $n$.
\item[{\rm (2.3)}] The three categories respectively consisting of objects with $\D$-injective, projective
and $\C$-injective dimensions at most $n$ coincide.
\end{enumerate}
\end{enumerate}
\end{thm}

As a consequence of Theorem \ref{thm-1.1}, we get that $\A$ is $n$-Gorenstein if and only if 
the Gorenstein projective dimension of any object in $\A$ is at most $n$,
if and only if the Gorenstein injective dimension of any object in $\A$ is at most $n$,
if and only $\A$ is right $n$-$(\mathcal{P}(\A),\mathcal{P}(\A))$-Gorenstein, if and only if 
$\A$ is left $n$-$(\mathcal{I}(\A),\mathcal{I}(\A))$-Gorenstein, and if and only if 
the two categories respectively consisting of objects with projective and injective dimensions
at most $n$ coincide. (Theorem \ref{thm-3.17}). This result extends \cite[Proposition VII.2.4(i)]{BR}.

In Section 4, we apply the results obtained in Section 3 to module categories. 
The study of generalized tilting modules, which were usually called Wakamatsu tilting modules, 
was initiated by Wakamatsu \cite{W1}. It follows from \cite[Corollary 3.2]{W3} that
${_RC}$ with $S=\End(_RC)$ is Wakamatsu tilting if and only if ${C_S}$ is Wakamatsu tilting with $R=\End(C_S)$,
and if and only if $_RC_S$ is semidualizing bimodule in the sense of \cite{HW}.
Many authors have studied the properties of Wakamatsu tilting modules and semidualizing bimodules
and related modules, see \cite{ATY,BR}, \cite{Fo}, \cite{HJ,HW}, \cite{Hu3}--\cite{HS}, 
\cite{LHX}, \cite{MR}, \cite{TW}--\cite{W3} and references therein.

Let $R$ be an arbitrary ring and $_RC$ a Wakamatsu tilting module with $S=\End(_RC)$.
We use $\mathcal{P}_C(R)$ and $\mathcal{F}_C(R)$ to denote the classes consisting of 
$C$-projective and $C$-flat left $R$-modules, respectively, and use   
$\mathcal{I}_C(S)$ and $\mathcal{FI}_C(S)$  to denote the classes consisting of 
$C$-injective and $C$-FP-injective left $S$-modules, respectively (see Section 2.2 for 
the definitions of these modules).
As an application of Theorem \ref{thm-1.1}, we get the following result.


\begin{thm} \label{thm-1.2}{\rm (Theorem \ref{thm-4.6})}
It holds that
\begin{enumerate}
\item[{\rm (1)}] 
For any $n\ge 0$, the following statements are equivalent.
\begin{enumerate}
\item[{\rm (1.1)}] The category of left $R$-modules is right $n$-$(\mathcal{P}_C(R),\mathcal{P}_C(R))$-Gorenstein.
\item[{\rm (1.2)}] The category of left $R$-modules is right $n$-$(\mathcal{P}_C(R),\mathcal{F}_C(R))$-Gorenstein.
\item[{\rm (1.3)}] The category of left $S$-modules is left $n$-$(\mathcal{I}_C(S),\mathcal{I}_C(S))$-Gorenstein.
\item[{\rm (1.4)}] The category of left $S$-modules is left $n$-$(\mathcal{I}_C(S),\mathcal{FI}_C(S))$-Gorenstein.
\item[{\rm (1.5)}] The $C$-Gorenstein projective dimension of any left $R$-module is at most $n$. 
\item[{\rm (1.6)}] The $C$-strong Gorenstein flat dimension of any left $R$-module is at most $n$.
\item[{\rm (1.7)}] The $C$-Gorenstein injective dimension of any left $S$-module is at most $n$.
\item[{\rm (1.8)}] The $C$-Gorenstein FP-injective dimension of any left $S$-module is at most $n$.
\end{enumerate}
\item[{\rm (2)}] 
If $C$ is faithful, then all the above and below conditions are equivalent.
\begin{enumerate}
\item[{\rm (2.1)}] The two categories respectively consisting of left $R$-modules with injective 
and $\mathcal{P}_C(R)$-projective dimensions at most $n~($and the category consisting of left $R$-modules 
with $\mathcal{F}_C(R)$-projective dimension at most $n)$ coincide.
\item[{\rm (2.2)}] The two categories respectively consisting of left $S$-modules with projective 
and $\mathcal{I}_C(S)$-injective dimensions at most $n~($and the category consisting of left $S$-modules 
with $\mathcal{FI}_C(S)$-injective dimension at most $n)$ coincide.
\end{enumerate}
\item[{\rm (3)}] 
If $R$ is a left Noetherian ring and $S$ is a right coherent ring, then all conditions in {\rm (1)} 
are equivalent to that both the FP-injective dimensions of $_RC$ and $C_S$ are at most $n$.
\end{enumerate}
\end{thm}

Note that the conditions in Theorem \ref{thm-1.2}(1) are not left-right symmetric in general (Remark \ref{rem-4.17}(1)),
and that the assumption ``$R$ is a left Noetherian ring and $S$ is a right coherent ring" 
in Theorem \ref{thm-1.2}(3) can not be weakened to ``$R$ is a left coherent ring and $S$ is a right coherent ring"
(Example \ref{exa-4.13}).

By Theorem \ref{thm-1.2}, we obtain several equivalent characterizations 
for the category of left $R$-modules and the ring $R$ to be $n$-Gorenstein (Corollary \ref{cor-4.7}).
As an immediate consequence of Theorem \ref{thm-1.2}, we also get that the suprema
of $C$-Gorenstein projective and $C$-strong Gorenstein flat dimensions of all left $R$-modules
and the suprema of $C$-Gorenstein injective and $C$-Gorenstein FP-injective dimensions
of all left $S$-modules are identical (Corollary \ref{cor-4.8}(1)). It extends some results in 
\cite{BM,HS,HMP,Wa}.

The Auslander class $\mathcal{A}_C(S)$ and Bass class $\mathcal{B}_C(R)$ with respect to 
a semidualizing bimodule ${_RC_S}$ (equivalently, a Wakamatsu tilting module $_RC$ with $S=\End(_RC)$)
originated from \cite{Fo} are now collectively known as Foxby classes,
and they are linked together by Foxby equivalence \cite{Fo,HW}
and duality pairs \cite{Hu4}. These two classes of modules play a
crucial role in studying the homological theory related to semidualizing bimodules and Wakamatsu tilting modules,
see \cite{ATY,Fo,HJ,HW}, \cite{Hu3}--\cite{HS}, \cite{LHX,MR}, \cite{TW}--\cite{W3} and references therein. 

We use $\pd_RC$ and $\pd_{S^{op}}C$
to denote the projective dimensions of $_RC$ and $C_S$, respectively.
The Wakamatsu tilting conjecture states that for artin algebras $R$ and $S$, it holds that
$\pd_RC=\pd_{S^{op}}C$ (cf. \cite{BR, MR}). This conjecture remains still open.
It was proved in \cite[Proposition 4.1]{TH3} that $\pd_RC=\pd_{S^{op}}C$ when both of them are finite,
which was proved in \cite[Proposition 7]{W1} when $R$ and $S$ are artin algebras.
As another application of Theorem \ref{thm-1.1}, we get the the following result, which greatly improves 
\cite[Theorems 4.2 and 4.10]{TH3}. Compare it with Theorem \ref{thm-1.2}.

\begin{thm} \label{thm-1.3} {\rm (Part of Theorem \ref{thm-4.15})}
For any $n\ge 0$, the following statements are equivalent.
\begin{enumerate}
\item[{\rm (1)}] $\pd_RC=\pd_{S^{op}}C\leq n$.
\item[{\rm (2)}] The category of left $R$-modules is left $n$-$(\mathcal{P}_C(R),\mathcal{P}_C(R))$-Gorenstein.
\item[{\rm (3)}] The category of left $S$-modules is right $n$-$(\mathcal{I}_C(S),\mathcal{I}_C(S))$-Gorenstein.
\item[{\rm (4)}] The $\mathcal{B}_C(R)$-injective dimension of any left $R$-module is at most $n$.
\item[{\rm (5)}] The $\mathcal{A}_C(S)$-projective dimension of any left $S$-module is at most $n$.
\item[$(i)^{op}$] The symmetric version of $(i)$ with $2\leq i\leq 5$.
\end{enumerate}
\end{thm}

As an immediate consequence, we get that the supremum of $\mathcal{A}_C(S)$-projective dimensions 
of all left $S$-modules and that of $\mathcal{B}_C(R)$-injective dimensions of all left $R$-modules are identical,
and they are left and right symmetric (Corollary \ref{cor-4.16}).
Over coherent semi-local rings, we establish the relation between the projective dimension of $C$
and the Bass injective dimension as well as the Auslander projective dimension of certain
semisimple module (Proposition \ref{prop-4.19}).
Then we obtain a necessary condition for the validity of the Wakamatsu tilting
conjecture (Corollary \ref{cor-4.20}).

\section {\bf Preliminaries}

In this section, we give some notions and preliminary results needed in the sequel.

\subsection{Relative homological dimensions and one-sided Gorenstein subcategories}

In this paper, $\A$ is an abelian category and all subcategories of $\A$ involved are full, 
and closed under isomorphisms. We use $\mathcal{P}(\A)$ and $\mathcal{I}(\A)$ to denote the subcategories 
of $\A$ consisting of projective and injective objects, respectively.

Let $\mathscr{X}$ be a subcategory of $\A$. We write
$${^\perp{\mathscr{X}}}:=\{A\in\A\mid\operatorname{Ext}^{\geq 1}_{\A}(A,X)=0 \mbox{ for any}\ X\in \mathscr{X}\},$$
$${{\mathscr{X}}^\perp}:=\{A\in\A\mid\operatorname{Ext}^{\geq 1}_{\A}(X,A)=0 \mbox{ for any}\ X\in \mathscr{X}\}.$$
The subcategory $\X$ is called {\it self-orthogonal} if $\X\subseteq{^\perp{\mathscr{X}}}$ 
(equivalently, $\X\subseteq{{\mathscr{X}}^\perp}$).
A sequence $\mathbb{E}$ in $\A$ is called {\it $\Hom_{\A}(\mathscr{X},-)$-exact}
(resp. {\it $\Hom_{\A}(-,\mathscr{X})$-exact}) if
$\Hom_{\A}(X,\mathbb{E})$ (resp. $\Hom_{\A}(\mathbb{E},X)$) is exact for any $X\in \mathscr{X}$.

Let $A\in\A$. The \textit{$\mathscr{X}$-projective dimension} 
$\mathscr{X}$-$\pd A$ of $A$ is defined as
$$\inf\{n\mid \text{there exists an exact sequence}\ 0 \to X_n \to \cdots \to X_1\to X_0 \to A\to 0$$
$$\text{in}\ \A \ \text{with all}\ X_i\ \text{in}\ \mathscr{X}\},$$
and set $\mathscr{X}$-$\pd A=\infty$ if no such integer exists.
Dually, the \textit{$\mathscr{X}$-injective dimension} $\mathscr{X}$-$\id A$ of $A$ is defined as
$$\inf\{n\mid \text{there exists an exact sequence}\ 0 \to A\to X^0\to X^1\to\cdots \to X^n\to 0$$
$$\text{in}\ \A \ \text{with all}\ X^i\ \text{in}\ \mathscr{X}\},$$
and set $\mathscr{X}$-$\id A=\infty$ if no such integer exists. For simplicity, we write
$$\pd A:=\mathcal{P}(\A)\text{-}\pd A\ \ \text{and}\ \ \id A:=\mathcal{I}(\A)\text{-}\id A.$$
For any $n\geq 0$, We use $\mathscr{X}\mbox{-}\pd^{\leq n}$ and $\mathscr{X}\mbox{-}\id^{\leq n}$
to denote the subcategories of $\A$ consisting of objects with $\mathscr{X}$-projective and
$\mathscr{X}$-injective dimensions at most $n$, respectively; and use 
$\mathscr{X}\mbox{-}\pd^{<\infty}$ and $\mathscr{X}\mbox{-}\id^{<\infty}$
to denote the subcategories of $\A$ consisting of objects with finite $\mathscr{X}$-projective and
$\mathscr{X}$-injective dimensions, respectively. For a subcategory $\Y$ of $\A$, we write
$$\mathscr{X}\mbox{-}\pd\Y:=\sup\{\mathscr{X}\mbox{-}\pd Y\mid Y\in\Y\}\ \ \text{and}\ \ 
\mathscr{X}\mbox{-}\id\Y:=\sup\{\mathscr{X}\mbox{-}\id Y\mid Y\in\Y\}.$$
Recall that $\X$ is called {\it resolving} in $\A$ if $\mathcal{P}(\A)\subseteq\X$ and 
$\X$ is closed under extensions and kernels of epimorphisms; dually, 
$\X$ is called {\it coresolving} in $\A$ if $\mathcal{I}(\A)\subseteq\X$ and 
$\X$ is closed under extensions and cokernels of monomorphisms.

\begin{lem}\label{lem-2.1}
Let $\X$ be a subcategory of $\mathscr{A}$ closed under direct summands.
\begin{enumerate}
\item[$(1)$] If $\X$ is resolving in $\mathscr{A}$,
then $\X$-$\pd^{\leq n}$ is closed under direct summands for any $n\geq 0$.
\item[$(2)$] If $\X$ is coresolving in $\mathscr{A}$,
then $\X$-$\id^{\leq n}$ is closed under direct summands for any $n\geq 0$.
\end{enumerate}
\end{lem}

\begin{proof}
If $\X$ is resolving (respectively, coresolving), then $\X$ is $\mathcal{P}(\mathscr{A})$-resolving
admitting a $\mathcal{P}(\mathscr{A})$-proper generator $\mathcal{P}(\mathscr{A})$ (respectively, $\mathcal{I}(\mathscr{A})$-coresolving
admitting an $\mathcal{I}(\mathscr{A})$-coproper cogenerator $\mathcal{I}(\mathscr{A})$) in the sense of \cite{Hu2}. Now the
assertions follow from \cite[Theorems 3.9 and 4.9]{Hu2}, respectively.
\end{proof}

\begin{df} \label{def-2.2}
{\rm (\cite{EJ})
Let $\X$ be a subcategory of $\A$. A morphism $f: X\to A$ in
$\A$ with $X\in\X$ and $A\in \A$ is called an {\it $\mathscr{X}$-precover} of $A$
if $\Hom_{\A}(X',f)$ is epic for any $X'\in\X$. 
The subcategory $\mathscr{X}$ is called {\it precovering} in $\A$ if any object
in $\mathscr{A}$ admits an $\mathscr{X}$-precover. Dually, the notions of an {\it $\mathscr{X}$-preenvelope}
and a {\it preenveloping subcategory} are defined.}
\end{df}

For subcategories $\mathscr{X}$ and $\mathscr{Y}$ of $\A$, we write
$$\widetilde{\res_{\mathscr{Y}}{\mathscr{X}}}:=\left\{A\in\A\mid  \begin{array}{c}
\mbox{there exists a $\Hom_{\A}(\mathscr{Y},-)$-exact exact sequence} \\
\cdots\to X_i \to \cdots \to X_1 \to X_0 \to A \to 0 \mbox{ in $\A$ with all $X_i$ in $\mathscr{X}$}
\end{array}\right\},$$
and
$$\widetilde{\cores_{\mathscr{Y}}{\mathscr{X}}}:=\left\{A\in\A\mid  \begin{array}{c}
\mbox{there exists a $\Hom_{\A}(-,\mathscr{Y})$-exact exact sequence} \\
0\to A \to X^0 \to X^1 \to \cdots \to X^i \to \cdots \mbox{ in $\A$ with all $X_i$ in $\mathscr{X}$}
\end{array}\right\}.$$
For simplicity, we write
$$\widetilde{\cores{\mathscr{X}}}:=\widetilde{\cores_{\mathscr{X}}{\mathscr{X}}}\ \ 
\text{and}\ \ \widetilde{\cores{\mathscr{X}}}:=\widetilde{\cores_{\mathscr{X}}{\mathscr{X}}}.$$

\begin{lem}\label{lem-2.3} 
Let $\X$ be a subcategory of $\A$. Assume that $M\in\widetilde{\cores\X}$ and $N\in\widetilde{\res\X}$, 
that is, there exist a $\Hom_\A(-,\X)$-exact exact sequence
$$0\to M\buildrel{f}\over\longrightarrow X^0\buildrel{f^0}\over\longrightarrow X^1
\buildrel{f^1}\over\longrightarrow\cdots\eqno{(2.1)}$$
and a $\Hom_\A(\X,-)$-exact exact sequence
$$\cdots\buildrel{g_1}\over\longrightarrow X_1\buildrel{g_0}\over\longrightarrow X_0
\buildrel{g}\over\longrightarrow N\to 0\eqno{(2.2)}$$
in $\A$ with all $X^i$ and $X_i$ in $\X$. 
Then $(2.1)$ is $\Hom_\A(-,N)$-exact if and only if $(2.2)$ is $\Hom_\A(M,-)$-exact. 
\end{lem}

\begin{proof}
From (2.1) and (2.2), we get the following commutative diagram: 
$$\xymatrix{
& \vdots\ar[d]^{(X^1,g_1)} & \vdots \ar[d]^{(X^,g_1)}  & \vdots \ar[d]^{(M,g_1)}  & \\
\cdots \ar[r]^{(f^1,X_1)\ \ \ } & (X^1,X_1) \ar[r]^{(f^0,X_1)}\ar[d]^{(X^1,g_0)} & (X^0,X_1) 
\ar[d]^{(X^0,g_0)} \ar[r]^{(f,X_1)} & (M,X_1) \ar[d]^{(M,g_0)} \ar[r] & 0\\
\cdots \ar[r]^{(f^1,X_0)\ \ \ } & (X^1,X_0) \ar[r]^{(f^0,X_0)}\ar[d]^{(X^1,g)} & (X^0,X_0) 
\ar[d]^{(X^0,g)} \ar[r]^{(f,X_0)} & (M,X_0) \ar[d]^{(M,g)} \ar[r] & 0\\
\cdots \ar[r]^{(f^1,N)\ \ \ } & (X^1,N) \ar[r]^{(f^0,N)}\ar[d] & (X^0,N) 
\ar[d] \ar[r]^{(f,N)} & (M,N) \ar[d] \ar[r] & 0\\
& 0 & 0 & 0, & & }$$
where $(-,-)=\Hom_\A(-,-)$. By assumption, all rows but the bottom one and all columns but the rightmost one 
are exact. Thus the bottom row is exact if and only if so is the rightmost column, that is, 
$(2.1)$ is $\Hom_\A(-,N)$-exact if and only if $(2.2)$ is $\Hom_\A(M,-)$-exact. 
\end{proof}

\begin{df}\label{def-2.4} {\rm  (\cite{SSW})
Let $\C$ be a subcategory of $\A$. The {\it Gorenstein subcategory}
$\mathcal{G}(\mathscr{C})$ of $\mathscr{A}$ relative to $\C$ is defined to be $\{G\in\A\mid$
there exists an exact sequence
$$\cdots \to C_1 \to C_0 \to C^0 \to C^1 \to \cdots$$ 
in $\A$ with all $C_i,C^i$ in $\C$, which is $\Hom_{\mathscr{A}}(\mathscr{C},-)$-exact and
$\Hom_{\mathscr{A}}(-,\mathscr{C})$-exact, such that $G\cong \Im(C_0\to C^0)\}$.}
\end{df}

Note that $\mathcal{G}(\mathcal{P}(\A))$ and $\mathcal{G}(\mathcal{I}(\A))$ are exactly 
Gorenstein projective and injective subcategories of $\A$, respectively (\cite{EJ}).

For a subcategory $\C$ of $\A$, following \cite[Lemma 5.7]{Hu1},
if $\C$ is self-orthogonal, then
$$\mathcal{G}(\C)=({^\perp\mathscr{C}}\cap\widetilde{\cores\mathscr{C}})\cap
({\mathscr{C}^\perp}\cap\widetilde{\res\mathscr{C}}).$$
Motivated by this, one-sided Gorenstein subcategories were introduced as follows.

\begin{df}\label{def-2.5}
{\rm Let $\mathscr{C}$ and $\mathscr{D}$ be subcategories of $\mathscr{A}$. Write 
$$r\mathcal{G}(\C):={^\perp\mathscr{C}}\cap\widetilde{\cores\mathscr{C}}\
(respectively,\ l\mathcal{G}(\C):={\mathscr{C}^\perp}\cap\widetilde{\res\mathscr{C}}),$$
$$r\mathcal{G}(\C,\D):={^\perp\D}\cap \widetilde{\cores_{\D}\mathscr{C}}\
(respectively,\ l\mathcal{G}(\C,\D):={\D^\perp}\cap \widetilde{\res_{\D}\mathscr{C}}).$$
We call $r\mathcal{G}(\C)$ (respectively, $l\mathcal{G}(\C)$) the 
{\it right} $($respectively, {\it left$)$ Gorenstein subcategory} of $\A$ relative to $\C$ (\cite{SZH}), 
and call $r\mathcal{G}(\C,\D)$ (respectively, $l\mathcal{G}(\C,\D)$) the 
{\it right} $($respectively, {\it left$)$ Gorenstein subcategory} of $\A$ relative to $\C$ and $\D$ (\cite{GW}).}
\end{df}

It is easy to see that if $\C\subseteq\C'$ and $\D\supseteq\D'$, then $r\mathcal{G}(\C,\D)\subseteq r\mathcal{G}(\C',\D')$
and $l\mathcal{G}(\C,\D)\subseteq l\mathcal{G}(\C',\D')$.

\subsection{Wakamatsu tilting modules}

In this paper, all rings are arbitrary associative rings with unit.
Let $R$ be a ring. We use $\Mod R$ to denote the category of left $R$-modules,
and use $\mathcal{P}(R)$, $\mathcal{F}(R)$ and $\mathcal{I}(R)$ to denote the subcategories of $\Mod R$
consisting of projective, flat and injective modules, respectively. For a left $R$-module $M$, 
we use $\pd_RM$, $\fd_RM$ and $\id_RM$ to denote the projective, flat and injective dimensions of $M$, respectively,
and use $\Add_RM$ ($\add_RM$) to denote the subcateory of $\Mod R$ consisting of 
direct summands of direct sums of (finite) copies of $M$.

\begin{df} \label{def-2.6} {\rm (\cite{W1,W3}). Let $R$ be a ring. A left $R$-module $C$
is called {\it generalized tilting} if the following conditions are satisfied.
\begin{enumerate}
\item[{\rm (1)}] ${_RC}$ admits a degreewise finite $R$-projective resolution.
\item[{\rm (2)}] $\Ext_{R}^{\geq 1}(C,C)=0$.
\item[{\rm (3)}] There exists a $\Hom_R(-,\add_RC)$-exact exact sequence
$$0\to{_RR}\to C^0\to C^1\to\cdots\to C^i\to\cdots$$
in $\Mod R$ with all $C^i$ in $\add_RC$.
\end{enumerate}}
\end{df}

Generalized tilting modules are usually called {\it Wakamatsu tilting modules}, see \cite{BR, MR}.
Let $C$ be a left $R$-module and $S=\End(_RC)$. It follows from \cite[Corollary 3.2]{W3} that
${_RC}$ is Wakamatsu tilting if and only if ${C_S}$ is Wakamatsu tilting with $R=\End(C_S)$,
and if and only if $_RC_S$ is semidualizing bimodule in the sense of \cite{ATY,HW}.
Typical examples of Wakamatsu tilting modules include the free module
of rank one and the dualizing module over a Cohen-Macaulay local ring.
For more examples of Wakamatsu tilting modules, the reader is referred to \cite{HW,TH2,W2}.

Recall from \cite{M,St} that a module $M\in\Mod R$ is called
{\it FP-injective} (or {\it absolutely pure}) if $M\in\mathcal{FI}(R)$, where
$\mathcal{FI}(R)=\{M\in\Mod R\mid\Ext_R^1(X,M)=0$ for any finitely presented
left $R$-module $X\}$.

In the following, let $R$ be an arbitrary ring, and let $_RC$ be a Wakamatsu tilting module with $S=\End(_RC)$.
We write
$$(-)_*:=\Hom(C,-),$$ and write
$$\mathcal{P}_C(R):=\{C\otimes_SP\mid P\in\mathcal{P}(S)\}\ \ \text{and}\ \
\mathcal{P}_C(S^{op}):=\{P'\otimes_RC\mid P'\in\mathcal{P}(R^{op})\},$$
$$\mathcal{I}_C(S):=\{I_*\mid I\in\mathcal{I}(R)\}\ \ \text{and}\ \ 
\mathcal{I}_C(R^{op}):=\{I'_*\mid I'\in\mathcal{I}(S^{op})\},$$
$$\mathcal{F}_C(R):=\{C\otimes_SF\mid F\in\mathcal{F}(S)\}\ \ \text{and}\ \ 
\mathcal{FI}_C(S):=\{M_*\mid M\in\mathcal{FI}(R)\}.$$
The modules in $\mathcal{P}_C(R)$ (respectively, $\mathcal{P}_C(S^{op})$), $\mathcal{F}_C(R)$, 
$\mathcal{I}_C(S)$ (respectively, $\mathcal{I}_C(R^{op})$) and $\mathcal{FI}_C(S)$ 
are called {\it $C$-projective}, {\it $C$-flat}, {\it $C$-injective} and {\it $C$-FP-injective}, respectively.
When ${_RC_S}={_RR_R}$, $C$-projective, $C$-flat, $C$-injective and $C$-FP-injective modules are exactly
projective, flat, injective and FP-injective modules, respectively.

\begin{lem}\label{lem-2.7} 
It holds that
$$\id\mathcal{P}_C(R)=\id\mathcal{F}_C(R)=\mathcal{I}_C(S){\text -}\id\mathcal{P}(S)=\mathcal{I}_C(S){\text -}\id\mathcal{F}(S).$$
\end{lem}

\begin{proof}
By \cite[Theorem 3.3(1)]{TH4}, we have
$$\id\mathcal{P}_C(R)=\mathcal{I}_C(S){\text -}\id\mathcal{P}(S)=\mathcal{I}_C(S){\text -}\id\mathcal{F}(S).$$
Since $\mathcal{F}(S)\subseteq\mathcal{A}_C(S)$ by \cite[Lemma 4.1]{HW}, we have 
$\id\mathcal{F}_C(R)=\mathcal{I}_C(S){\text -}\id\mathcal{F}(S)$ by \cite[Theorem 3.5(3)]{TH2}.
\end{proof}

The first two notions in the following definition were introduced by Holm and J{\o}rgensen \cite{HJ} for
commutative rings, and their non-commutative versions were given in \cite{LHX}.

\begin{df}\label{def-2.8} 
{\rm \begin{enumerate}
\item[]
\item[$(1)$] 
A module $M\in\Mod R$ is called {\it $C$-Gorenstein projective} if $M\in r\mathcal{G}(\mathcal{P}_C(R))$.
\item[$(2)$] 
A module $N\in\Mod S$ is called {\it $C$-Gorenstein injective} if $M\in l\mathcal{G}(\mathcal{I}_C(S))$.
\item[$(3)$] (\cite{HS})
A module $M\in\Mod R$ is called {\it $C$-strong Gorenstein flat} if $M\in r\mathcal{G}(\mathcal{P}_C(R),\mathcal{F}_C(R))$.
\item[$(4)$] 
A module $N\in\Mod S$ is called {\it $C$-Gorenstein FP-injective} if $M\in l\mathcal{G}(\mathcal{I}_C(S),\mathcal{FI}_C(S))$.
\end{enumerate}}
\end{df}

When $_RC_S={_RR_R}$, $C$-Gorenstein projective, $C$-Gorenstein injective, $C$-strong Gorenstein flat
and $C$-Gorenstein FP-injective moduels are exactly Gorenstein projective, Gorenstein injective, strong Gorenstein flat
and Gorenstein FP-injective modules, respectively (\cite{DLM,EJ,H,MD}).

\begin{df} \label{def-2.9} {\rm (\cite{HW})
\begin{enumerate}
\item[(1)] The {\it Auslander class} $\mathcal{A}_{C}(S)$ with respect to $C$
consists of all modules $N$ in $\Mod S$ satisfying the following conditions:
\begin{enumerate}
\item[(A1)] $\Tor^{S}_{\geq 1}(C,N)=0$.
\item[(A2)] $\Ext_R^{\geq 1}(C,C\otimes_SN)=0$.
\item[(A3)] The canonical evaluation homomorphism
$$\mu_N: N\rightarrow(C\otimes_SN)_*$$ defined by $\mu_N(y)(x)=x\otimes y$
for any $y\in N$ and $x\in C$ is an isomorphism.
\end{enumerate}
\end{enumerate}
\begin{enumerate}
\item[(2)] The {\it Bass class} $\mathcal{B}_C(R)$ with respect to $C$
consists of all modules $M$ in $\Mod R$ satisfying the following conditions:
\begin{enumerate}
\item[(B1)] $\Ext_R^{\geq 1}(C,M)=0$.
\item[(B2)] $\Tor^{S}_{\geq 1}(C,M_*)=0$.
\item[(B3)] The canonical evaluation homomorphism
$$\theta_M:C\otimes_SM_*\rightarrow M$$ defined by $\theta_M(x\otimes f)=f(x)$
for any $x\in C$ and $f\in M_*$ is an isomorphism.
\end{enumerate}
\item[(3)] The {\it Auslander class} $\mathcal{A}_C(R^{op})$ in $\Mod R^{op}$
and the {\it Bass class} $\mathcal{B}_C(S^{op})$ in $\Mod S^{op}$ are defined symmetrically.
\end{enumerate}}
\end{df}

The Auslander and Bass classes are certain right and left Gorenstein subcategories, respectively
(see Section 4.2 for details).

\section{General results}

In this section,  $\mathscr{A}$ is an abelian category.

\subsection{Relative projective and injective dimensions}

\begin{lem} \label{lem-3.1}
Let $\U$ be a subcategory of $\A$. Assume that one of the following two conditions is satisfied:
\begin{enumerate}
\item[$(a)$] $\U$ is self-orthogonal and closed under kernels of epimorphisms;
\item[$(b)$] $\U$ is precovering and resolving and $\U$ is closed under direct summands.
\end{enumerate}
Then for any $M\in\A$ and $n\geq 0$, the following statements are equivalent.
\begin{enumerate}
\item[{\rm (1)}] $\U$-$\pd M\leq n$.
\item[{\rm (2)}] There exists a $\Hom_\A(\U,-)$-exact exact sequence
$$0\to U_n\to\cdots\to U_1\to U_0\to M\to 0$$
in $\A$ with all $U_i$ in $\U$.
\end{enumerate}
\end{lem}

\begin{proof}
$(2)\Longrightarrow (1)$ 
It is clear.

$(1)\Longrightarrow (2)$ 
If the condition $(a)$ is satisfied, then any exact sequence
$$0\to U_n\to\cdots\to U_1\to U_0\to M\to 0$$
in $\A$ with all $U_i$ in $\U$ is $\Hom_\A(\U,-)$-exact, and the assertion follows.

Now suppose that the condition $(b)$ is satisfied. We proceed by induction on $n$. The case for $n=0$ is trivial.
Suppose $n\geq 1$. Then there exists an exact sequence
$$0\to K'\to U'_0\to M\to 0$$
in $\A$ with $U'_0\in\U$ and $\U$-$\pd K'\leq n$ by (1). It yields that any $\U$-precover of $M$ is epic.
Since $\U$ is precovering in $\A$, there exists a $\Hom_\A(\U,-)$-exact exact sequence
$$0\to K\to U_0\to M\to 0$$
in $\A$ with $U_0\in\U$. Consider the following pull-back diagram:
$$\xymatrix{& & 0 \ar@{-->}[d] & 0 \ar[d]& &\\
& & K \ar@{==}[r] \ar@{-->}[d] & K \ar[d]& &\\
0 \ar@{-->}[r] & K' \ar@{==}[d] \ar@{-->}[r] & X \ar@{-->}[d]\ar@{-->}[r] & U_0\ar[d] \ar@{-->}[r] & 0\\
0 \ar[r] & K' \ar[r] & U'_0 \ar@{-->}[d] \ar[r] & M \ar[d] \ar[r] & 0\\
& & 0 & 0. & & }$$
Since $\U$ is resolving, applying \cite[Theorem 3.2]{Hu3} to the middle row in the above diagram yields 
$\U$-$\pd X\leq n-1$. On the other hand, since the middle column in the above diagram is $\Hom_\A(\U,-)$-exact
by \cite[Lemma 2.4(1)]{Hu1}, it splits and $X\cong K\oplus U'_0$. Thus $\U$-$\pd K\leq n$ by
\cite[Corollary 3.9]{Hu2}, and the assertion follows by induction.
\end{proof}

Dually, we have the following result.

\begin{lem} \label{lem-3.2}
Let $\E$ be a subcategory of $\A$. Assume that one of the following two conditions is satisfied:
\begin{enumerate}
\item[$(a)$] $\E$ is self-orthogonal and closed under cokernels of monomorphisms;
\item[$(b)$] $\E$ is preenveloping and coresolving and $\E$ is closed under direct summands.
\end{enumerate}
Then for any $M\in\A$ and $n\geq 0$, the following statements are equivalent.
\begin{enumerate}
\item[{\rm (1)}] $\E$-$\id M\leq n$.
\item[{\rm (2)}] There exists a $\Hom_\A(-,\E)$-exact exact sequence
$$0\to M\to E^0\to E^1\to\cdots\to E^n\to 0$$
in $\A$ with all $E^i$ in $\E$.
\end{enumerate}
\end{lem}

\begin{rem}\label{rem-3.3}
{\rm Both the condition ``$\U$ is closed under kernels of epimorphisms" in Lemma \ref{lem-3.1}$(a)$ and 
the condition ``$\E$ is closed under cokernels of monomorphisms" in Lemma \ref{lem-3.2}$(a)$ are superfluous.
However, In view that they are necessary to prove the rest results in this subsection, 
for the sake simplicity of subsequent statements we include them there.}
\end{rem}

\begin{prop} \label{prop-3.4}
Let $\U$ and $\E$ be subcategories of $\A$. If one of the conditions $(a)$ and $(b)$ in Lemma \ref{lem-3.1}
is satisfied, then for any $n\geq 0$, the following statements are equivalent.
\begin{enumerate}
\item[{\rm (1)}] $\U$-$\pd\E\leq n$.
\item[{\rm (2)}] $\E$-$\id^{<\infty}\subseteq\U$-$\pd^{\leq n}$.
\end{enumerate}
\end{prop}

\begin{proof}
$(2)\Longrightarrow (1)$ 
It is trivial.

$(1)\Longrightarrow (2)$ 
Let  $M\in\E$-$\id^{<\infty}$. Then there exists the following exact sequence
$$0\to M\to E^0\to E^1\to\cdots\to E^m\to 0$$
in $\A$ with $m\geq 0$ and all $E^i$ in $\E$ is exact. By (1) and Lemma \ref{lem-3.1}, 
there exists a $\Hom_\A(\U,-)$-exact exact sequence   
$$0\to U_n^i\to\cdots\to U_1^i\to U_0^i\to E^i\to 0$$
in $\A$ with all $U_j^i$ in $\U$ for any $0\leq i\leq m$ and $0\leq j\leq n$. 
Applying \cite[Corollary 3.3(1)]{Hu1} yields the following two exact sequences:
$$0\to\oplus_{i=0}^mU_{i+n}^i\to\cdots\to\oplus_{i=0}^mU_{i+2}^i\to\oplus_{i=0}^mU_{i+1}^i\to U\to M\to 0,\eqno{(3.1)}$$
$$0\to U\to\oplus_{i=0}^mU_{i}^i\to\oplus_{i=1}^mU_{i-1}^i\to\cdots\to U_{0}^{m-1}\oplus U_1^m\to U_0^m\to 0,\eqno{(3.2)}$$
where $U_i^j=0$ whenever $i\geq n+1$. Since $\U$ is additive and closed under kernels of epimorphisms, we have $U\in\U$
by (3.2). It follows from (3.1) that $\U$-$\pd M\leq n$ and $M\in\U$-$\pd^{\leq n}$.
\end{proof}

Dually, we have the following result.

\begin{prop} \label{prop-3.5}
Let $\U$ and $\E$ be subcategories of $\A$. If one of the conditions $(a)$ and $(b)$ in Lemma \ref{lem-3.2}
is satisfied, then for any $n\geq 0$, the following statements are equivalent.
\begin{enumerate}
\item[{\rm (1)}] $\E$-$\id\U\leq n$.
\item[{\rm (2)}] $\U$-$\pd^{<\infty}\subseteq\E$-$\id^{\leq n}$.
\end{enumerate}
\end{prop}

As a consequence of Propositions \ref{prop-3.4} and \ref{prop-3.5}, we get the following result.

\begin{prop} \label{prop-3.6}
Let $\U$ and $\E$ be subcategories of $\A$. 
If one of the conditions $(a)$ and $(b)$ in Lemma \ref{lem-3.1} and
one of the conditions $(a)$ and $(b)$ in Lemma \ref{lem-3.2} are satisfied.
\begin{enumerate}
\item[{\rm (1)}] 
If $\V$ is a subcategory of $\A$ containing $\U$,
then for any $n\geq 0$, the following statements are equivalent.
\begin{enumerate}
\item[{\rm (1.1)}] $\U$-$\pd\E\leq n$ and $\E$-$\id\V\leq n$.
\item[{\rm (1.2)}] $\U$-$\pd^{\leq n}=\E$-$\id^{\leq n}=\V$-$\pd^{\leq n}$.
\end{enumerate}
\item[{\rm (2)}] 
If $\F$ is a subcategory of $\A$ containing $\E$,
then for any $n\geq 0$, the following statements are equivalent.
\begin{enumerate}
\item[{\rm (2.1)}] $\E$-$\id\U\leq n$ and $\U$-$\pd\F\leq n$.
\item[{\rm (2.2)}] $\E$-$\id^{\leq n}=\U$-$\pd^{\leq n}=\F$-$\id^{\leq n}$.
\end{enumerate}
\end{enumerate}
\end{prop}

\begin{proof}
(1) By Propositions \ref{prop-3.4} and \ref{prop-3.5}, we have that $\U$-$\pd\E\leq n$ and $\E$-$\id\V\leq n$
if and only if 
$$\V\text{-}\pd^{\leq n}\subseteq\V\text{-}\pd^{<\infty}\subseteq\E\text{-}\id^{\leq n}
\subseteq\E\text{-}\id^{<\infty}\subseteq\U\text{-}\pd^{\leq n}.$$
Since $\U\subseteq\V$ by assumption, we have $\U\text{-}\pd^{\leq n}\subseteq\V\text{-}\pd^{\leq n}$, 
and the assertion follows.

(2) It is dual to (1).
\end{proof}

\begin{rem}\label{rem-3.7}
{\rm In the case when all ``$\leq n$" in Propositions \ref{prop-3.4}--\ref{prop-3.6} are replaced by ``$<\infty$", 
all these results still hold true.}
\end{rem}

\subsection{One-sided Gorenstein categories}

In this subsection, $\mathscr{A}$ is an abelian category having enough projective and injective objects.
We first introduce one-sided Gorensrein categories as follows.

\begin{df} \label{def-3.8}
{\rm Let $\mathscr{C}$ and $\mathscr{D}$ be subcategories of $\mathscr{A}$, and let $n\geq 0$.
\begin{enumerate}
\item[{\rm (1)}] $\A$ is called a {\it right $n$-$(\C,\D)$-Gorenstein category} if 
$$\mathscr{C}{\text -}\pd \mathcal{I}(\mathscr{A})\leq n\ \text{and}\ \id \mathscr{D}\leq n.$$
In particular, if $\C=\D$, then a right $n$-$(\C,\D)$-Gorenstein category is called 
a {\it right $n$-$\C$-Gorenstein category} for short.
\item[{\rm (2)}] $\A$ is called a {\it left $n$-$(\C,\D)$-Gorenstein category} if 
$$\mathscr{C}{\text -}\id \mathcal{P}(\mathscr{A})\leq n\ \text{and}\ \pd \mathscr{D}\leq n.$$
In particular, if $\C=\D$, then a left $n$-$(\C,\D)$-Gorenstein category is called 
a {\it left $n$-$\C$-Gorenstein category} for short.
\end{enumerate}}
\end{df}

If $\C\subseteq\C'$ and $\D\supseteq\D'$, then a right $n$-$(\C,\D)$-Gorenstein category
and a left $n$-$(\C,\D)$-Gorenstein category are right $n$-$(\C',\D')$-Gorenstein and 
left $n$-$(\C',\D')$-Gorenstein, respectively.
We list some examples of one-sided Gorenstein categories as follows. See Section 4 for more examples.

\begin{exa} \label{exa-3.9}
{\rm 
Recall from \cite[Definition VII.2.1]{BR} that $\A$ is called {\it Gorenstein} if $\id\mathcal{P}(\A)<\infty$ and 
$\pd\mathcal{I}(\A)<\infty$. It follows from \cite[Proposition VII.1.3(vi)]{BR} that $\A$ is Gorenstein 
if and only if $\id\mathcal{P}(\A)=\pd\mathcal{I}(\A)<\infty$. If $\id\mathcal{P}(\A)=\pd\mathcal{I}(\A)\leq n<\infty$, 
then we call $\A$ {\it $n$-Gorenstein}. It is easy to see that $\A$ is $n$-Gorenstein if and only if 
it is right $n$-$\mathcal{P}(\A)$-Gorenstein, and if and only if it is left $n$-$\mathcal{I}(\A)$-Gorenstein.}
\end{exa}

From now on, assume that $\mathscr{A}$ satisfies the AB4 and AB4$^{*}$ conditions, that is,
$\mathscr{A}$ has arbitrary direct sums and products which are exact.

\begin{lem} \label{lem-3.10}
For any $M\in\mathscr{A}$, there exists an exact sequence
$$0\to M\oplus X\to P\oplus I\to M\oplus X\to 0$$
in $\mathscr{A}$ with $P$ projective and $I$ injective.
\end{lem}

\begin{proof}
Let $M\in\mathscr{A}$, and let
$$\cdots\to P_i\to\cdots\to P_1\to P_0\to M\to 0$$
and
$$0\to M\to I^0\to I^1\to\cdots\to I^i\to\cdots$$
be a projective resolution and an injective coresolution of $M$ in $\mathscr{A}$, respectively.
Decompose them respectively into short exact sequences:
$$0\to U_{i}\to P_{i}\to U_{i-1}\to 0, $$
and
$$0\to V^{i-1}\to I^{i}\to V^{i}\to 0,$$
where $U_i=\Im(P_{i+1}\to P_{i})$ and $V^i=\Im(I^{i}\to I^{i+1})$ with $i\geq 0$, and in particular, $U_{-1}=V^{-1}=M$.
Then we get the following two exact sequences:
$$0\to \oplus_{i\geq 0}U_i\to\oplus_{i\geq 0}P_i\to M\oplus(\oplus_{i\geq 0}U_i)\to 0$$
and
$$0\to M\oplus(\Pi_{i\geq 0}V^i)\to\Pi_{i\geq 0}I^i\to\Pi_{i\geq 0}V^i\to 0.$$
It induces the following exact sequence:
$$0\to M\oplus(U\oplus V)\to P\oplus I\to M\oplus(U\oplus V)\to 0,$$
where $U=\oplus_{i\geq 0}U_i$, $V=\Pi_{i\geq 0}V^i$, $P:=\oplus_{i\geq 0}P_i\in\mathcal{P}(\A)$ 
and $I:=\Pi_{i\geq 0}I^i\in\mathcal{I}(\A)$.
\end{proof}

We write
$$\pd^{\leq n}:=\{A\in\A\mid\pd A\leq n\}\ \text{and}\ \id^{\leq n}:=\{A\in\A\mid\id A\leq n\}.$$
The following is the main result in this section.

\begin{thm} \label{thm-3.11}
Let $\mathscr{C}$ and $\mathscr{D}$ be subcategories of $\mathscr{A}$ with $\C$  
closed under direct summands. Assume that
$$\mathscr{C}\subseteq\mathscr{D}\cap{^{\bot}\mathscr{D}}\ \ and\ \
\mathcal{P}(\mathscr{A})\subseteq r\mathcal{G}(\mathscr{C},\mathscr{D})$$
and $n\ge 0$. Consider the following conditions.
\begin{enumerate}
\item[{\rm (1)}] $\A$ is right $n$-$(\C,\D)$-Gorenstein 
$($that is, $\mathscr{C}$-$\pd \mathcal{I}(\mathscr{A})\leq n$ and $\id \mathscr{D}\leq n)$.
\item[{\rm (2)}] $r\mathcal{G}(\mathscr{C},\mathscr{D})$-$\pd M\leq n$ for any $M\in\mathscr{A}$.
\item[{\rm (3)}] $r\mathcal{G}(\mathscr{C},\mathscr{D})$-$\pd \mathcal{I}(\mathscr{A})\leq n$ and
$\id \mathscr{D}\leq n$.
\item[{\rm (4)}] $\mathscr{D}$-$\pd^{\leq n}=\id^{\leq n}=\mathscr{C}$-$\pd^{\leq n}$.
\item[{\rm (5)}] $\mathscr{D}$-$\pd^{\leq n}\subseteq\id^{\leq n}\subseteq 
r\mathcal{G}(\mathscr{C},\mathscr{D})$-$\pd^{\leq n}$.
\end{enumerate}
It holds that $(4)\Longrightarrow (1)\Longleftrightarrow (2)\Longleftrightarrow (3)\Longleftarrow (5)$.
\begin{enumerate}
\item[{\rm (i)}] 
If $\mathscr{C}$ is closed under kernels of epimorphisms, then all the conditions $(1)$--$(4)$ are equivalent.
\item[{\rm (ii)}] 
If $r\mathcal{G}(\mathscr{C},\mathscr{D})$ is precovering and resolving, 
then all the conditions $(1)$--$(3)$  and $(5)$ are equivalent.
\end{enumerate}
\end{thm}

\begin{proof}
By \cite[Proposition 4.6(1)]{GW}, we have $(1)\Longleftrightarrow (3)$.

$(2)\Longrightarrow (3)$
Let $M\in\mathscr{A}$. By (2), there exists an exact sequence
$$0\to G_n\to\cdots\to G_1\to G_0\to M\to 0$$
in $\mathscr{A}$ with all $G_i$ in $r\mathcal{G}(\mathscr{C},\mathscr{D})$. For any $D\in\mathscr{D}$,
applying the functor $\Hom_\mathscr{A}(-,D)$ to the above exact sequence yields
$$\Ext_\mathscr{A}^{n+1}(M,D)\cong\Ext_\mathscr{A}^1(G_n,D)=0,$$
and thus $\id D\leq n$.

$(3)\Longrightarrow (2)$
Let $M\in\mathscr{A}$. By Lemma \ref{lem-3.10}, there exists an exact sequence
$$0\to M\oplus X\to P\oplus I\to M\oplus X\to 0$$
in $\mathscr{A}$ with $P$ projective and $I$ injective.
By the horseshoe lemma, we get the following commutative diagram with exact columns and rows:
$$\xymatrix{& 0 \ar[d] & 0 \ar@{-->}[d] & 0 \ar[d] & \\
0 \ar@{-->}[r] & K_n \ar[d] \ar@{-->}[r] & G_n \ar@{-->}[d] \ar@{-->}[r] & K_n \ar[d] \ar@{-->}[r] & 0\\
0 \ar@{-->}[r] & Q_{n-1} \ar[d] \ar@{-->}[r] & Q_{n-1}\oplus Q_{n-1}
\ar@{-->}[d] \ar@{-->}[r] & Q_{n-1} \ar[d] \ar@{-->}[r] & 0\\
&\vdots \ar[d] & \vdots\ar@{-->}[d] & \vdots \ar[d] & \\
0 \ar@{-->}[r] & Q_{1} \ar[d] \ar@{-->}[r] & Q_{1}\oplus Q_{1}
\ar@{-->}[d] \ar@{-->}[r] & Q_{1} \ar[d] \ar@{-->}[r] & 0\\
0 \ar@{-->}[r] & Q_{0} \ar[d] \ar@{-->}[r] & Q_{0}\oplus Q_{0}
\ar@{-->}[d] \ar@{-->}[r] & Q_{0} \ar[d] \ar@{-->}[r] & 0\\
0 \ar[r] & M\oplus X \ar[d] \ar[r] & P\oplus I \ar@{-->}[d] \ar[r]
& M\oplus X \ar[d] \ar[r] & 0\\
& 0  & 0  & 0 & }$$
in $\mathscr{A}$ with all $Q_i$ projective. Since $P\in r\mathcal{G}(\mathscr{C},\mathscr{D})$ and
$r\mathcal{G}(\mathscr{C},\mathscr{D})$-$\pd I\leq n$ by assumption,
we have $r\mathcal{G}(\mathscr{C},\mathscr{D})$-$\pd (P\oplus I)\leq n$.
Since $\mathcal{P}(\mathscr{A})\subseteq r\mathcal{G}(\mathscr{C},\mathscr{D})$ by assumption,
it follows from \cite[Theorem 3.4]{GW} that $r\mathcal{G}(\mathscr{C},\mathscr{D})$ is resolving.
Then $G_n\in r\mathcal{G}(\mathscr{C},\mathscr{D})$ by \cite[Lemma 3.1(1)]{Hu3}.

Since $\id \mathscr{D}\leq n$ by (3),
according to the leftmost or rightmost column in the above diagram,
it is easy to get $K_n\in{^{\bot}\mathscr{D}}$ by dimension shifting. Then we get the following
$\Hom_\mathscr{A}(-,\mathscr{D})$-exact exact sequence:
$$\cdots\to G_n\to\cdots \to G_n\to G_n\to G_n\to\cdots\to G_n\to\cdots,$$
in which the image of each homomorphism is $K_n$. By \cite[Proposition 3.12]{GW}, we have
$K_n\in r\mathcal{G}(\mathscr{C},\mathscr{D})$.
According to the leftmost or rightmost column in the above diagram again, we have
$r\mathcal{G}(\mathscr{C},\mathscr{D})$-$\pd(M\oplus X)\leq n$. Since $r\mathcal{G}(\mathscr{C},\mathscr{D})$
is closed under direct summands by \cite[Theorem 3.6]{GW}, it follows from
Lemma \ref{lem-2.1}(1) that $r\mathcal{G}(\mathscr{C},\mathscr{D})$-$\pd M\leq n$.

It is easy to get $(4)\Longrightarrow (1)$ and $(5)\Longrightarrow (3)$.

Note that $\mathcal{I}(\A)$ is self-orthogonal and closed under cokernels of monomorphisms.

Assume that $\mathscr{C}$ is closed under kernels of epimorphisms. Since $\C$ is self-orthogonal by assumption,
the implication $(1)\Longrightarrow (4)$ follows by putting $\U=\C$, $\V=\D$ and $\E=\mathcal{I}(\A)$ in Proposition \ref{prop-3.6}(1).

Assume that $r\mathcal{G}(\mathscr{C},\mathscr{D})$ is precovering and resolving. By \cite[Theorem 3.6]{GW}, 
we have that $r\mathcal{G}(\mathscr{C},\mathscr{D})$ is closed under direct summands. If (3) holds true,
then $r\mathcal{G}(\mathscr{C},\mathscr{D})$-$\pd \mathcal{I}(\mathscr{A})\leq n$, and hence 
$$\id^{\leq n}\subseteq r\mathcal{G}(\mathscr{C},\mathscr{D})\text{-}\pd^{\leq n}$$
by putting $\U=r\mathcal{G}(\mathscr{C},\mathscr{D})$ and $\E=\mathcal{I}(\A)$ in Proposition \ref{prop-3.4}. 
On the other hand, we have 
$$\D\text{-}\pd^{\leq n}\subseteq\id^{\leq n}$$
by putting $\U=\D$ and $\E=\mathcal{I}(\A)$ in Proposition \ref{prop-3.5}, and thus the assertion (5) follows.
This proves $(3)\Longrightarrow (5)$. 
\end{proof}

The following result is a special case of Theorem \ref{thm-3.11}.

\begin{cor} \label{cor-3.12}
Let $\mathscr{C}$ be a self-orthogonal subcategory of $\mathscr{A}$ 
closed under direct summands. If
$\mathcal{P}(\mathscr{A})\subseteq r\mathcal{G}(\mathscr{C})$, 
then for any $n\ge 0$, the following statements are equivalent.
\begin{enumerate}
\item[{\rm (1)}] $\A$ is right $n$-$\C$-Gorenstein  
$($that is, $\mathscr{C}$-$\pd \mathcal{I}(\mathscr{A})\leq n$ and $\id \mathscr{C}\leq n)$.
\item[{\rm (2)}] $r\mathcal{G}(\mathscr{C})$-$\pd M\leq n$ for any $M\in\mathscr{A}$.
\item[{\rm (3)}] $r\mathcal{G}(\mathscr{C})$-$\pd \mathcal{I}(\mathscr{A})\leq n$ and
$\id \mathscr{C}\leq n$.
\end{enumerate}
Furthermore, if $\mathscr{C}$ is closed kernels of epimorphisms, then all the above and below conditions are equivalent.
\begin{enumerate}
\item[{\rm (4)}] $\mathscr{C}$-$\pd^{\leq n}=\id^{\leq n}$.
\end{enumerate}
\end{cor}

We also have the following corollary.

\begin{cor} \label{cor-3.13}
Let $\mathscr{D}$ be a subcategory of $\mathscr{A}$ containing $\mathcal{P}(\A)$. 
Then for any $n\ge 0$, the following statements are equivalent.
\begin{enumerate}
\item[{\rm (1)}] $\A$ is right $n$-$(\mathcal{P}(\A),\D)$-Gorenstein  
$($that is, $\pd \mathcal{I}(\mathscr{A})\leq n$ and $\id \mathscr{D}\leq n)$.
\item[{\rm (2)}] $r\mathcal{G}((\mathcal{P}(\A),\D))$-$\pd M\leq n$ for any $M\in\mathscr{A}$.
\item[{\rm (3)}] $r\mathcal{G}((\mathcal{P}(\A),\D))$-$\pd \mathcal{I}(\mathscr{A})\leq n$ and
$\id \mathscr{D}\leq n$.
\item[{\rm (4)}] $\mathscr{D}$-$\pd^{\leq n}=\id^{\leq n}=\pd^{\leq n}$.
\end{enumerate}
\end{cor}

\begin{proof}
It is trivial that $\mathcal{P}(\A)\subseteq r\mathcal{G}((\mathcal{P}(\A),\D))$. 
Putting $\C=\mathcal{P}(\A)$ in Theorem \ref{thm-3.11}, the assertion follows.
\end{proof}

The following result is dual to Theorem \ref{thm-3.11}, we omit its proof. 

\begin{thm} \label{thm-3.14}
Let $\mathscr{C}$ and $\mathscr{D}$ be subcategories of $\mathscr{A}$ such that
$\mathscr{C}$ is closed under direct summands. Assume that
$$\mathscr{C}\subseteq\mathscr{D}\cap{\mathscr{D}^{\bot}}\ \ and\ \
\mathcal{I}(\mathscr{A})\subseteq l\mathcal{G}(\mathscr{C},\mathscr{D})$$
and $n\ge 0$. Consider the following conditions.
\begin{enumerate}
\item[{\rm (1)}] $\A$ is left $n$-$(\C,\D)$-Gorenstein 
$($that is, $\mathscr{C}$-$\id \mathcal{P}(\mathscr{A})\leq n$ and $\pd \mathscr{D}\leq n)$.
\item[{\rm (2)}] $l\mathcal{G}(\mathscr{C},\mathscr{D})$-$\id M\leq n$ for any $M\in\mathscr{A}$.
\item[{\rm (3)}] $l\mathcal{G}(\mathscr{C},\mathscr{D})$-$\id \mathcal{P}(\mathscr{A})\leq n$
and $\pd \mathscr{D}\leq n$.
\item[{\rm (4)}] $\mathscr{D}$-$\id^{\leq n}=\pd^{\leq n}=\mathscr{C}$-$\id^{\leq n}$.
\item[{\rm (5)}] $\mathscr{D}$-$\id^{\leq n}\subseteq\pd^{\leq n}\subseteq 
l\mathcal{G}(\mathscr{C},\mathscr{D})$-$\id^{\leq n}$.
\end{enumerate}
It holds that $(4)\Longrightarrow (1)\Longleftrightarrow (2)\Longleftrightarrow (3)\Longleftarrow (5)$.
\begin{enumerate}
\item[{\rm (i)}] 
If $\mathscr{C}$ is closed under cokernels of monomorphisms, then all the conditions $(1)$--$(4)$ are equivalent.
\item[{\rm (ii)}] 
If $l\mathcal{G}(\mathscr{C},\mathscr{D})$ is preenveloping and coresolving, 
then all the conditions $(1)$--$(3)$  and $(5)$ are equivalent.
\end{enumerate}
\end{thm}

The following result is a special case of Theorem \ref{thm-3.14}.

\begin{cor} \label{cor-3.15}
Let $\mathscr{C}$ be a self-orthogonal subcategory of $\mathscr{A}$  
closed under direct summands. If
$\mathcal{I}(\mathscr{A})\subseteq l\mathcal{G}(\mathscr{C})$, 
then for any $n\ge 0$, the following statements are equivalent.
\begin{enumerate}
\item[{\rm (1)}] $\A$ is left $n$-$\C$-Gorenstein 
$($that is, $\mathscr{C}$-$\id \mathcal{P}(\mathscr{A})\leq n$ and $\pd \mathscr{C}\leq n)$.
\item[{\rm (2)}] $l\mathcal{G}(\mathscr{C})$-$\id M\leq n$ for any $M\in\mathscr{A}$.
\item[{\rm (3)}] $l\mathcal{G}(\mathscr{C})$-$\id \mathcal{P}(\mathscr{A})\leq n$
and $\pd \mathscr{C}\leq n$.
\end{enumerate}
Furthermore, if $\mathscr{C}$ is closed cokernels of monomorphisms, then all the above and below conditions are equivalent.
\begin{enumerate}
\item[{\rm (4)}] $\mathscr{C}$-$\id^{\leq n}=\pd^{\leq n}$.
\end{enumerate}
\end{cor}

We also have the following corollary, which is dual to Corollary \ref{cor-3.13}.

\begin{cor} \label{cor-3.16}
Let $\mathscr{D}$ be a subcategory of $\mathscr{A}$ containing $\mathcal{I}(\A)$. 
Then for any $n\ge 0$, the following statements are equivalent.
\begin{enumerate}
\item[{\rm (1)}] $\A$ is left $n$-$(\mathcal{I}(\A),\D)$-Gorenstein  
$($that is, $\id \mathcal{P}(\mathscr{A})\leq n$ and $\pd \mathscr{D}\leq n)$.
\item[{\rm (2)}] $l\mathcal{G}((\mathcal{I}(\A),\D))$-$\id M\leq n$ for any $M\in\mathscr{A}$.
\item[{\rm (3)}] $l\mathcal{G}((\mathcal{I}(\A),\D))$-$\id \mathcal{P}(\mathscr{A})\leq n$ and
$\pd \mathscr{D}\leq n$.
\item[{\rm (4)}] $\mathscr{D}$-$\id^{\leq n}=\pd^{\leq n}=\id^{\leq n}$.
\end{enumerate}
\end{cor}


We write
$$\Gpd M:=r{\mathcal{G}}(\mathcal{P}(\A)){\text -}\pd M\ \ 
\text{and}\ \ \Gid N:=l{\mathcal{G}}(\mathcal{I}(\A)){\text -}\id N.$$
As a consequence of Corollaries \ref{cor-3.12} and \ref{cor-3.15} (or Corollaries \ref{cor-3.13} and \ref{cor-3.16}), 
we obtain some equivalent characterization of $n$-Gorenstein categories as follows, which extends 
\cite[Proposition VII.2.4(i)]{BR}.

\begin{thm} \label{thm-3.17}
For any $n\ge 0$, the following statements are equivalent.
\begin{enumerate}
\item[{\rm (1)}] $\A$ is $n$-Gorenstein $($that is, $\pd \mathcal{I}(\A)=\id\mathcal{P}(\A)\leq n)$.
\item[{\rm (2)}] $\Gpd M\leq n$ for any $M\in\A$.
\item[{\rm (3)}] $\Gid N\leq n$ for any $N\in\A$.
\item[{\rm (4)}] $\Gpd\mathcal{I}(\A)\leq n$ and $\id\mathcal{P}(\A)\leq n$.
\item[{\rm (5)}] $\Gid\mathcal{P}(\A)\leq n$ and $\pd\mathcal{I}(\A)\leq n$.
\item[{\rm (6)}] $\pd^{\leq n}=\id^{\leq n}$.
\end{enumerate}
\end{thm}

\begin{proof}
Putting $\mathscr{C}=\mathcal{P}(\A)$ in Corollary \ref{cor-3.12}, 
we get $(1)\Longleftrightarrow (2)\Longleftrightarrow (4)\Longleftrightarrow (6)$. 
Putting $\mathscr{C}=\mathcal{I}(\A)$ in Corollary \ref{cor-3.15}, 
we get $(1)\Longleftrightarrow (3)\Longleftrightarrow (5)\Longleftrightarrow (6)$.
\end{proof}

Finally, we list other cases that satisfy the conditions in Corollaries \ref{cor-3.12} and \ref{cor-3.15}.

\begin{rem}\label{rem-3.18}
{\rm Let $R$ be a ring and let $(\U,\V)$ be a hereditary cotorsion pair in $\Mod R$ with $\C:=\U\cap\V$. 
\begin{enumerate}
\item[{\rm (1)}]
If $\mathcal{P}(R)\subseteq\C$, then $\mathcal{P}(R)\subseteq r\mathcal{G}(\mathscr{C})$
and the conditions in Corollary \ref{cor-3.12} are satisfied.
For example, by \cite[Corollary 3.4(1)]{CS}, we have that $(\mathcal{G}(\mathcal{P}(R)),{\mathcal{G}(\mathcal{P}(R))^{\bot}})$
is a hereditary cotorsion pair and $\mathcal{P}(R)\subseteq\C:=\mathcal{G}(\mathcal{P}(R))\cap{\mathcal{G}(\mathcal{P}(R))^{\bot}}$.
\item[{\rm (2)}] If $\mathcal{I}(R)\subseteq\C$, then $\mathcal{I}(R)\subseteq l\mathcal{G}(\mathscr{C})$
and the conditions in Corollary \ref{cor-3.15} are satisfied.
For example, by \cite[Theorem 5.6]{SS}, we have that $({^{\bot}\mathcal{G}(\mathcal{I}(R))},\mathcal{G}(\mathcal{I}(R)))$
is a hereditary cotorsion pair and $\mathcal{I}(R)\subseteq\C:={^{\bot}\mathcal{G}(\mathcal{I}(R))}\cap\mathcal{G}(\mathcal{I}(R))$.
\end{enumerate}}
\end{rem}

\section{Categories of interest}

In this section, $R$ is an arbitrary ring and $_RC$ is Wakamatsu tilting module with $S=\End(_RC)$.
For any subcategory $\X$ of $\Mod R$ and $n\geq 0$, we use $\X\text{-}\pd^{\leq n}(R)$ and $\X\text{-}\id^{\leq n}(R)$
to denote the subcategories of $\Mod R$ consisting of modules with $\X$-projective and $\X$-injective dimensions
at most $n$, respectively. We write 
$$(-)^+:=\Hom_{\mathbb{Z}}(-,\mathbb{Q}/\mathbb{Z}),$$
where $\mathbb{Z}$ is the additive group of integers and $\mathbb{Q}$ is the additive group of rational numbers.


\begin{prop} \label{prop-4.1}
Let $\X$ be a coresolving subcategory of $\Mod R$ closed under direct summands, and 
let $\Y$ be a resolving subcategory of $\Mod R^{op}$ closed under direct summands.
Assume that the following conditions are satisfied:
\begin{enumerate}
\item[$(a)$]  For a module $X\in\Mod R$, it holds that  $X\in\X$ if and only if $X^+\in\Y$;
\item[$(b)$] For a module $Y\in\Mod R^{op}$, it holds that $Y\in\Y$ if and only if $Y^+\in\X$.
\end{enumerate} 
Then we have
\begin{enumerate}
\item[{\rm (1)}] $\X$-$\id M=\Y$-$\pd M^+$ for any $M\in\Mod R$.
\item[{\rm (2)}] $\Y$-$\pd N=\X$-$\id N^+$ for any $N\in\Mod R^{op}$.
\end{enumerate}
\end{prop}

\begin{proof}
(1) Let $M\in\Mod R$. We first prove $\Y$-$\pd M^+\leq\X$-$\id M$. If $\X$-$\id M=\infty$, 
then the assertion follows trivially. Now suppose that $\X$-$\id M=n<\infty$ and
$$0\to M\to X^0\to X^1\to\cdots\to X^n\to 0$$
is an exact sequence in $\Mod R$ with all $X^i$ in $\X$. It yields the following exact sequence
$$0\to {X^n}^+\to\cdots\to {X^1}^+\to {X^0}^+\to M^+\to 0$$
in $\Mod R^{op}$. By assumption, we have that all ${X^i}^+$ are in $\Y$ and 
$\Y$-$\pd M^+\leq n=\X$-$\id M$.

In the following, we prove $\X$-$\id M\leq\Y$-$\pd M^+$. If $\Y$-$\pd M^+=\infty$, 
then the assertion follows trivially. Now suppose that $\Y$-$\pd M^+=n<\infty$ and
$$0\to M\to I^0\to I^1\to\cdots\to I^{n-1}\to X^n\to 0$$
is an exact sequence in $\Mod R$ with all $I^i$ injective. It yields the following exact sequence
$$0\to {X^n}^+\to {I^{n-1}}^+\to\cdots\to {I^1}^+\to {I^0}^+\to M^+\to 0$$
in $\Mod R^{op}$. Since all $I^i$ are in $\X$,
we have that all ${I^i}^+$ are in $\Y$.
Since $\Y$ is resolving and closed under direct summands, we have ${X^n}^+\in\Y$
by \cite[Lemma 3.1(1)]{Hu3}. Then $X^n\in\X$ by assumption, and thus 
$\X$-$\id M\leq n=\Y$-$\pd M^+$.

(2) It is dual to (1).
\end{proof}

Recall from \cite{EJ,H} an exact sequence 
$$0\to K\to M\to L\to 0$$
in $\Mod R$ is called {\it pure exact} if 
$$0\to A\otimes_RK\to A\otimes_RM\to A\otimes_RL\to 0$$
is exact for any $A\in\Mod R^{op}$. In this case, $K$ and $L$ are called a {\it pure submodule} 
and a {\it pure quotient} of $M$, respectively. As a consequence of Proposition \ref{prop-4.1}, we obtain the following result.

\begin{cor} \label{cor-4.2}
It holds that
\begin{enumerate}
\item[{\rm (1)}] Under the assumptions in Proposition \ref{prop-4.1}, 
for any $n\ge 0$, both $\X$-$\id^{\leq n}(R)$ and $\Y$-$\pd^{\leq n}(R^{op})$
are closed under pure submodules and pure quotients.
\item[{\rm (2)}] For any $n\ge 0$, both $\mathcal{B}_C(R)$-$\id^{\leq n}$ and $\mathcal{A}_C(S)$-$\pd^{\leq n}$
are closed under pure submodules and pure quotients.
\end{enumerate}
\end{cor}

\begin{proof}
(1) Let $$0\to K\to M \to L\to 0$$
be a pure exact sequence in $\Mod R$ with $\X$-$\id M\leq n$.
Then by \cite[Proposition 5.3.8]{EJ}, the induced exact sequence
$$0\to L^+\to M^+\to K^+\to 0$$
splits and both $K^+$ and $L^+$ are direct summands of $M^+$.
By Proposition \ref{prop-4.1}(1), we have $\Y$-$\pd M^+\leq n$.
Since $\Y$ is resolving and closed under direct summands,
we have that $\Y$-$\pd^{\leq n}$ is closed under direct summands by Lemma \ref{lem-2.1}(1). 
It follows that $\Y$-$\pd K^+\leq n$ and $\Y$-$\pd L^+\leq n$.
Thus $\X$-$\id K\leq n$ and $\X$-$\id L\leq n$
by Proposition \ref{prop-4.1}(1) again.

Similarly, we have that $\Y$-$\pd^{\leq n}(R^{op})$ is closed under pure submodules and pure quotients.

(2) By \cite[Theorem 6.2 and Proposition 4.2(a)]{HW}, we have that 
both $\mathcal{B}_C(R)$ and $\mathcal{B}_C(S^{op})$ are coresolving and closed under direct summands, and 
both $\mathcal{A}_C(S)$ and $\mathcal{A}_C(R^{op})$ are resolving and closed under direct summands.
On the other hand, by \cite[Proposition 3.2]{Hu4}, we have 
\begin{enumerate}
\item[$(a)$] 
a module $B\in\Mod R$ (respectively, $\Mod S^{op}$), it holds that $B\in\mathcal{B}_C(R)$ 
(respectively, $\mathcal{B}_C(S^{op})$) 
if and only if $B^+\in\mathcal{A}_{C}(R^{op})$ (respectively, $\mathcal{A}_C(S)$);
\item[$(b)$]
a module $A\in\Mod R^{op}$ (respectively, $\Mod S$), it holds that $A\in\mathcal{A}_{C}(R^{op})$ 
(respectively, $\mathcal{A}_C(S)$)
if and only if $A^+\in\mathcal{B}_C(R)$ (respectively, $\mathcal{B}_C(S^{op})$).
\end{enumerate}
Now the assertion follows from (1).
\end{proof} 

For a left $R$-module $M$, we use $E(M)$
to denote the injective envelope of $M$, and write
$$\FPid_RM:=\mathcal{FI}(R){\text -}\id M.$$
The following result shows that the assumption ``$R$ is a left coherent ring"
in \cite[Proposition 3.3 and Theorem 3.4]{Hu5} is superfluous. 

\begin{prop}\label{prop-4.3} 
\begin{enumerate}
It holds that
\item[$(1)$] If $M\in\mathcal{FI}(R)$, then $M^+$ is a direct summand of $E(M)^+$ and 
${M_*}^+$ is a direct summand of ${E(M)_*}^+$.
\item[$(2)$] $\mathcal{FI}(R)\subseteq\mathcal{B}_C(R)$ and $\mathcal{FI}_C(S)$-$\id^{<\infty}\subseteq\mathcal{A}_C(S)\subseteq{^{\bot}\mathcal{I}_C(S)}$.
\item[$(3)$] $\FPid_RC\otimes_SN\leq\mathcal{FI}_C(S)$-$\id N$ for any $N\in\Mod S$; the equality holds true if $N\in\mathcal{A}_C(S)$.
\end{enumerate}
\end{prop}

\begin{proof}
Let $M\in\mathcal{FI}(R)$. Then $M$ is a pure submodule of $E(M)$.

(1) By using an argument similar to that in the proof of Corollary \ref{cor-4.2}(1), 
we get that $M^+$ is a direct summand of $E(M)^+$, 
and hence $M^+\otimes_RC$ is a direct summand of $E(M)^+\otimes_RC$. Since
${M_*}^+\cong M^+\otimes_RC$ and ${E(M)_*}^+\cong E(M)^+\otimes_RC$ by 
\cite[Lemma 1.16(c)]{GT}, it follows that ${M_*}^+$ is a direct summand of ${E(M)_*}^+$.

(2)  Since $E(M)\in\mathcal{B}_C(R)$ by \cite[Lemma 4.1]{HW}, we have $M\in\mathcal{B}_C(R)$
by Corollary \ref{cor-4.2}(2), and thus $\mathcal{FI}(R)\subseteq\mathcal{B}_C(R)$. Then by \cite[Proposition 4.1]{HW},
we have $\mathcal{FI}_C(S)\subseteq\mathcal{A}_C(S)$. It follows from \cite[Theorem 6.2]{HW} 
and \cite[Corollary 3.5(2)]{TH3} that 
$\mathcal{FI}_C(S)$-$\id^{<\infty}\subseteq\mathcal{A}_C(S)\subseteq{^{\bot}\mathcal{I}_C(S)}$.

(3) When $R$ is a left coherent ring, the assertion (3) was proved in \cite[Theorem 3.4]{Hu5} and its proof
depends on that the containment $\mathcal{FI}(R)\subseteq\mathcal{B}_C(R)$ holds true. In view that 
$\mathcal{FI}(R)\subseteq\mathcal{B}_C(R)$ over arbitrary rings by (2),  
the argument in the proof of \cite[Theorem 3.4]{Hu5} is valid in our setting.  
\end{proof}

The following observation is useful.

\begin{prop}\label{prop-4.4} 
Let $\mathscr{H}$ be a subclass of $\widetilde{\res\mathcal{P}_C(R)}$, 
and let $\mathscr{T}$ be a subclass of $\widetilde{\cores\mathcal{I}_C(S)}$. It holds that
\begin{enumerate}
\item[$(1)$]
$\mathcal{P}(R)\subseteq r\mathcal{G}(\mathcal{P}_C(R),\mathscr{H})$ and $\mathcal{I}(S)\subseteq l\mathcal{G}(\mathcal{I}_C(S),\mathscr{T})$.
\item[$(2)$]
If $\mathcal{P}_C(R)\subseteq\mathscr{H}$, then $r\mathcal{G}(\mathcal{P}_C(R),\mathscr{H})$ is resolving. 
If $\mathcal{I}_C(S)\subseteq\mathscr{T}$, then $l\mathcal{G}(\mathcal{I}_C(S),\mathscr{T})$ is coresolving. 
\end{enumerate}
\end{prop}

\begin{proof}
(1) Let $P\in\mathcal{P}(R)$ and $I\in\mathcal{I}(S)$. Since $r\mathcal{G}(\mathcal{P}_C(R))$ is resolving 
and $l\mathcal{G}(\mathcal{I}_C(S))$ is coresolving (cf. \cite[Remark 4.4(3)]{Hu3}), 
there exists a $\Hom_R(-,\mathcal{P}_C(R))$-exact exact sequence
$$0\to P\to Q^0\to Q^1\to\cdots\to\cdots Q^i\to\cdots\eqno{(4.1)}$$
in $\Mod R$ with all $Q^i$ in $\mathcal{P}_C(R)$, and there exists a $\Hom_S(\mathcal{I}_C(S),-)$-exact exact sequence
$$\cdots\to E_i\to\cdots\to E_1\to E_0\to I\to 0\eqno{(4.2)}$$
in $\Mod S$ with all $E_i$ in $\mathcal{I}_C(S)$. 
On the other hand, since $\mathscr{H}\subseteq\widetilde{\res\mathcal{P}_C(R)}$ 
and $\mathscr{T}\subseteq\widetilde{\cores\mathcal{I}_C(S)}$ by assumption,
we have that for any $H\in\mathscr{H}$ and $T\in\mathscr{T}$, 
there exists a $\Hom_R(\mathcal{P}_C(R),-)$-exact exact sequence
$$\cdots\to Q_i\to\cdots\to Q_1\to Q_0\to H\to 0\eqno{(4.3)}$$
in $\Mod R$ with all $Q_i$ in $\mathcal{P}_C(R)$, and there exists a $\Hom_S(-,\mathcal{I}_C(S))$-exact exact sequence
$$0\to T\to E^0\to E^1\to\cdots\to\cdots E^i\to\cdots\eqno{(4.4)}$$
in $\Mod S$ with all $E^i$ in $\mathcal{I}_C(S)$.
It is trivial that (4.3) is $\Hom_R(P,-)$-exact and (4.4) is $\Hom_S(-,I)$-exact. 
Then (4.1) is $\Hom_R(-,H)$-exact and (4.2) is $\Hom_S(T,-)$-exact by Lemma \ref{lem-2.3}.
So $P\in\widetilde{\cores_{\mathscr{H}}\mathcal{P}_C(R)}$ and $I\in\widetilde{\res_{\mathscr{T}}\mathcal{I}_C(S)}$, and hence 
$$P\in{^{\bot}\mathscr{H}}\cap\widetilde{\cores_{\mathscr{H}}\mathcal{P}_C(R)}=r\mathcal{G}(\mathcal{P}_C(R),\mathscr{H})$$
and $$I\in{\mathscr{T}^{\bot}}\cap\widetilde{\res_{\mathscr{T}}\mathcal{I}_C(S)}=l\mathcal{G}(\mathcal{I}_C(S),\mathscr{T}).$$

(2) If $\mathcal{P}_C(R)\subseteq\mathscr{H}$ and $\mathcal{I}_C(S)\subseteq\mathscr{T}$, then 
$r\mathcal{G}(\mathcal{P}_C(R),\mathscr{H})$ is closed under extensions and kernels of epimorphisms
and $l\mathcal{G}(\mathcal{I}_C(S),\mathscr{T})$ is closed under extensions and cokernels of monomorphisms
by \cite[Theorem 3.4]{GW} and its dual result, respectively. Now both assertions follow from (1).
\end{proof}

We write
$$\mathcal{I}_C(R^{op})^+:=\{E^+\mid E\in\mathcal{I}_C(R^{op})\}.$$
For simplicity, we write
$$\mathcal{GF}_C(R):=r\mathcal{G}(\mathcal{F}_C(R),\mathcal{I}_C(R^{op})^+)\ \ \text{and}\ \ 
\mathcal{PGF}_C(R):=r\mathcal{G}(\mathcal{P}_C(R),\mathcal{I}_C(R^{op})^+),$$
$$\mathcal{SGF}_C(R):=r\mathcal{G}(\mathcal{P}_C(R),\mathcal{F}_C(R))\ \text{(the class of $C$-strong Gorenstein flat left $R$-modules)},$$
$$\mathcal{GFI}_C(S):=l\mathcal{G}(\mathcal{I}_C(S),\mathcal{FI}_C(S))\ \text{(the class of $C$-Gorenstein FP-injective left $S$-modules)}.$$

\begin{cor}\label{cor-4.5} 
It holds that 
\begin{enumerate}
\item[$(1)$]
$\mathcal{P}(R)\subseteq\mathcal{PGF}_C(R)\subseteq\mathcal{GF}_C(R)$.
\item[$(2)$]
Both $\mathcal{SGF}_C(R)$ and $r\mathcal{G}(\mathcal{P}_C(R),\mathcal{B}_C(R))$ are resolving.
\item[$(3)$]
Both $\mathcal{GFI}_C(S)$ and $l\mathcal{G}(\mathcal{I}_C(S),\mathcal{A}_C(S))$ are coresolving.
\end{enumerate}
\end{cor}

\begin{proof}
We have the following facts:
\begin{enumerate}
\item[{\rm (i)}] Since $\mathcal{I}_C(R^{op})\subseteq\mathcal{A}_C(R^{op})$ by \cite[Corollary 6.1]{HW}, 
it follows from \cite[Proposition 3.2(1)]{Hu4} and \cite[Theorem 6.1]{HW} that 
$\mathcal{I}_C(R^{op})^+\subseteq\mathcal{B}_C(R)\subseteq\widetilde{\res\mathcal{P}_C(R)}$.
\item[{\rm (ii)}] $\mathcal{P}_C(R)\subseteq\mathcal{F}_C(R)\subseteq\mathcal{B}_C(R)\subseteq\widetilde{\res\mathcal{P}_C(R)}$
by \cite[Corollary 6.1 and Theorem 6.1]{HW}.
\item[{\rm (iii)}] $\mathcal{I}_C(S)\subseteq\mathcal{FI}_C(S)\subseteq\mathcal{A}_C(S)\subseteq\widetilde{\cores\mathcal{I}_C(S)}$ by 
Proposition \ref{prop-4.3}(2) and \cite[Theorem 2]{HW}.
\end{enumerate}
Now all assertions follow from Proposition \ref{prop-4.4}.
\end{proof}

Note that $\mathcal{GF}_C(R):=r\mathcal{G}(\mathcal{F}_C(R),\mathcal{I}_C(R^{op})^+)$ and 
$\mathcal{PGF}_C(R):=r\mathcal{G}(\mathcal{P}_C(R),\mathcal{I}_C(R^{op})^+)$
are exactly the classes of {\it $C$-Gorenstein flat} 
and {\it $C$-projectively coresolved Gorenstein flat} left $R$-modules, respectively (\cite{HS}). 
When $_RC_S={_RR_R}$, they are exactly the classes of Gorenstein flat 
and projectively coresolved Gorenstein flat left $R$-modules, respectively (\cite{EJ,SS}).

\subsection{Usual $C$-Gorenstein subcategories}

Following the usual customary notation, we write
$$\GCpd_RM:=r{\mathcal{G}}(\mathcal{P}_C(R)){\text -}\pd M\ \  
\text{and}\ \ \GCid_SN:=l{\mathcal{G}}(\mathcal{I}_C(S)){\text -}\id N,$$
$$\SGFCpd_RM:=\mathcal{SGF}_C(R){\text -}\pd M\ \  
\text{and}\ \ \GFPCid_SN:=\mathcal{GFI}_C(S){\text -}\id N.$$

A Wakamatsu tilting module $_RC$ with $S=\End(_RC)$ (equivalently, a semidualizing bimodule ${_RC_S}$)
is called {\it faithful} if the following conditions are satisfied: (i) if $M\in \Mod R$ satisfies $M_*=0$, then $M=0$;
and (ii) if $N\in \Mod S^{op}$ satisfies $N_*=0$, then $N=0$ (\cite{HW}).
In the following result, the implications $(1.1)\Longleftrightarrow (1.2)\Longleftrightarrow (2.1)\Longleftrightarrow (2.2)$ 
have been obtained in \cite[Theorem 5.4(1)]{HS}.

\begin{thm} \label{thm-4.6}
It holds that
\begin{enumerate}
\item[{\rm (i)}]
For any $n\ge 0$, the following statements are equivalent.
\begin{enumerate}
\item[{\rm (1.1)}] $\Mod R$ is right $n$-$\mathcal{P}_C(R)$-Gorenstein 
$($that is, $\mathcal{P}_C(R)$-$\pd \mathcal{I}(R)\leq n$ and $\id\mathcal{P}_C(R)\leq n)$.
\item[{\rm (1.2)}] $\GCpd_RM\leq n$ for any $M\in\Mod R$.
\item[{\rm (1.3)}] $\GCpd\mathcal{I}(R)\leq n$ and $\id\mathcal{P}_C(R)\leq n$.
\item[{\rm (2.1)}] $\Mod S$ is left $n$-$\mathcal{I}_C(S)$-Gorenstein 
$($that is, $\mathcal{I}_C(S)$-$\id\mathcal{P}(S)\leq n$ and $\pd\mathcal{I}_C(S)\leq n)$.
\item[{\rm (2.2)}] $\GCid_SN\leq n$ for any $N\in\Mod S$.
\item[{\rm (2.3)}] $\GCid \mathcal{P}(S)\leq n$ and $\pd \mathcal{I}_C(S)\leq n$.
\item[{\rm (3.1)}] $\Mod R$ is right $n$-$(\mathcal{P}_C(R),\mathcal{F}_C(R))$-Gorenstein 
$($that is, $\mathcal{P}_C(R)$-$\pd \mathcal{I}(R)\leq n$ and $\id\mathcal{F}_C(R)\leq n)$.
\item[{\rm (3.2)}] $\SGFCpd_RM\leq n$ for any $M\in\Mod R$.
\item[{\rm (3.3)}] $\SGFCpd\mathcal{I}(R)\leq n$ and $\id\mathcal{F}_C(R)\leq n$.
\item[{\rm (4.1)}] $\Mod S$ is left $n$-$(\mathcal{I}_C(S),\mathcal{FI}_C(S))$-Gorenstein 
$($that is, $\mathcal{I}_C(S)$-$\id\mathcal{P}(S)\leq n$ and $\pd\mathcal{FI}_C(S)$ $\leq n)$.
\item[{\rm (4.2)}] $\GFPCid_SN\leq n$ for any $N\in\Mod S$.
\item[{\rm (4.3)}] $\GFPCid\mathcal{P}(S)\leq n$ and $\pd \mathcal{FI}_C(S)\leq n$.
\end{enumerate}
\item[{\rm (ii)}]
If the Wakamatsu tilting module $C$ is faithful, then all conditions in {\rm (i)} and the below conditions are equivalent.
\begin{enumerate}
\item[{\rm (1.4)}] $\id^{\leq n}(R)=\mathcal{P}_C(R)$-$\pd^{\leq n}$.
\item[{\rm (2.4)}] $\pd^{\leq n}(S)=\mathcal{I}_C(S)$-$\id^{\leq n}$.
\item[{\rm (3.4)}] $\mathcal{F}_C(R)$-$\pd^{\leq n}=\id^{\leq n}(R)=\mathcal{P}_C(R)$-$\pd^{\leq n}$.
\item[{\rm (4.4)}] $\mathcal{FI}_C(S)$-$\id^{\leq n}=\pd^{\leq n}(S)=\mathcal{I}_C(S)$-$\id^{\leq n}$.
\end{enumerate}
\item[{\rm (iii)}]
If $R$ is a left Noetherian ring and $S$ is a right coherent ring, 
then all conditions in {\rm (i)} and the below condition are equivalent.
\begin{enumerate}
\item[{\rm (5)}] $\FPid_RC\leq n$ and $\FPid_{S^{op}}C\leq n$.
\end{enumerate}
\item[{\rm (iv)}]
If $R$ is a left Noetherian ring and $S$ is a right Noetherian ring, 
then all conditions in {\rm (i)} and the below condition are equivalent.
\begin{enumerate}
\item[{\rm (6)}] $\id_RC\leq n$ and $\id_{S^{op}}C\leq n$.
\end{enumerate}
\end{enumerate}
\end{thm}

\begin{proof}
Note that the following assertions hold true:
\begin{enumerate}
\item[{\rm (a)}] Both $\mathcal{P}_C(R)$ and $\mathcal{I}_C(S)$ are self-orthogonal by \cite[Lemma 2.7]{HS}.
\item[{\rm (b)}] $r{\mathcal{G}}(\mathcal{P}_C(R))$ is resolving and $l{\mathcal{G}}(\mathcal{I}_C(S))$ 
is coresolving (cf. \cite[Remark 4.4(3)]{Hu3}).
\item[{\rm (c)}] If the Wakamatsu tilting module $_RC$ with $S=\End(_RC)$ is faithful, then $\mathcal{P}_C(R)$ 
is closed under kernels of epimorphisms and $\mathcal{I}_C(S)$ is closed under cokernels of monomorphisms 
by \cite[Corollary 6.4]{HW}.
\end{enumerate}
Putting $\C=\D=\mathcal{P}_C(R)$ in Theorem \ref{thm-3.11}, 
we get $(1.1)\Longleftrightarrow (1.2)\Longleftrightarrow (1.3)\ (\Longleftrightarrow (1.4)$ if $C$ is faithful). 
Putting $\C=\D=\mathcal{I}_C(S)$ in Theorem \ref{thm-3.14}, we get 
$(2.1)\Longleftrightarrow (2.2)\Longleftrightarrow (2.3)\ (\Longleftrightarrow (2.4)$ if $C$ is faithful).
By \cite[Theorem 5.4(1)]{HS}, we have $(1.2)\Longleftrightarrow (2.2)$.

By Corollary \ref{cor-4.5}, we have 
$\mathcal{P}(R)\subseteq\mathcal{SGF}_C(R)$ and $\mathcal{I}(S)\subseteq\mathcal{GFI}_C(S)$. Note that 
$$\mathcal{P}_C(R)\subseteq\mathcal{F}_C(R)\cap{^{\bot}\mathcal{F}_C(R)}\ \text{and}\ 
\mathcal{I}_C(S)\subseteq\mathcal{FI}_C(S)\cap{\mathcal{FI}_C(S)^{\bot}}$$
by \cite[Lemma 2.5(1)]{TH4} and Proposition \ref{prop-4.3}(2), respectively.
Putting $\C=\mathcal{P}_C(R)$ and $\D=\mathcal{F}_C(R)$ in Theorem \ref{thm-3.11}, 
we get $(3.1)\Longleftrightarrow (3.2)\Longleftrightarrow (3.3)\ (\Longleftrightarrow (3.4)$ if $C$ is faithful).   
Putting $\C=\mathcal{I}_C(S)$ and $\D=\mathcal{FI}_C(S)$ in Theorem \ref{thm-3.14}, we get 
$(4.1)\Longleftrightarrow (4.2)\Longleftrightarrow (4.3)\ (\Longleftrightarrow (4.4)$ if $C$ is faithful)

By Lemma \ref{lem-2.7}, we get $(1.1)\Longleftrightarrow (3.1)$. It is trivial that $(4.1)\Longrightarrow (2.1)$.

$(2.1)+(3.1)\Longrightarrow (4.1)$ By (2.1), it suffices to prove $\pd_SN\leq n$ for any $N\in\mathcal{FI}_C(S)$.

Let $N\in\mathcal{FI}_C(S)$. Then $N=M_*$ for some $M\in\mathcal{FI}(R)$. By Proposition \ref{prop-4.3}(2), we have
$M\in\mathcal{B}_C(R)$ and $N\in\mathcal{A}_C(S)$. It holds that
\begin{align*}
& \mathcal{F}_C(R)\text{-}\pd M=\fd_SM_*\ \ \ \text{(by \cite[Theorem 3.5(1)]{TH2})}\\
= & \id_{S^{op}}{M_*}^+\leq \id_{S^{op}}{E(M)_*}^+\ \ \ \text{(by \cite[Theorem 2.1]{F} and Proposition \ref{prop-4.3}(1))}\\
= & \fd_SE(M)_*\leq\pd_SE(M)_*\leq n.\ \ \ \text{(by \cite[Theorem 2.1]{F} and (2.1))}
\end{align*}
Since $\id\mathcal{F}_C(R)\leq n$ by (3.1), we have $\id_RM\leq n$. Since $M\in\mathcal{B}_C(R)$, we have $C\otimes_SM_*\cong M$.
Notice that $N\in\mathcal{A}_C(S)$, it follows from \cite[Theorem 3.5(3)]{TH2} that 
$$\mathcal{I}_C(S)\text{-}\id N=\id_RC\otimes_SN=\id_RC\otimes_SM_*=\id_RM\leq n.$$
Since $\pd\mathcal{I}_C(S)\leq n$ by (2.1), we have $\pd_SN\leq n$. This proves (i) and (ii).

The assertion (iv) is a special case of (iii). In the following, we prove (iii).

Let $R$ be a left Noetherian ring and $S$ a right coherent ring. By \cite[Lemma 4.1]{HW}, we have 
$\mathcal{I}(R)\subseteq\mathcal{B}_C(R)$. Then 
$$\FPid_{S^{op}}C=\fd\mathcal{I}_C(S)=\mathcal{F}_C(R){\text -}\pd\mathcal{I}(R) \eqno{(4.5)}$$
by \cite[Lemma 2.1(2)]{ZD} and \cite[Theorem 3.5(1)]{TH2}.
 
$(1.1)\Longrightarrow (5)$ By the assertion (1.1), we have $\FPid_RC=\id_RC\leq n$. 
Then by the equality (4.5) and the assertion (1.1) again, we have 
$$\FPid_{S^{op}}C=\mathcal{F}_C(R){\text -}\pd\mathcal{I}(R)\leq \mathcal{P}_C(R){\text -}\pd\mathcal{I}(R)\leq n.$$

$(5)\Longrightarrow (1.1)$ By the assertion (5), we have $\id_RC=\FPid_RC\leq n$. 
Since $\mathcal{P}_C(R)=\Add_RC$ by \cite[Proposition 2.4(a)]{LHX}, we have
$\id\mathcal{P}_C(R)\leq n$ by \cite[Theorem 1.1]{B}. On the other hand, since
$\mathcal{F}_C(R){\text -}\pd\mathcal{I}(R)=\FPid_{S^{op}}C\leq n$ by the equality (4.5) and the assertion (5),
it follows from \cite[Lemma 5.19(1)]{HS} that $\mathcal{P}_C(R){\text -}\pd\mathcal{I}(R)\leq n$.
\end{proof}

In general, the conditions in Theorem \ref{thm-4.6}(i) are not left-right symmetric, see Remark \ref{rem-4.17}(1) below.

When ${_RC_S=_RR_R}$, we write
$$\Gpd_RM:=\GCpd_RM\ \ \text{and}\ \ \Gid_RM:=\GCid_RM,$$
$$\SGFpd_RM:=\SGFCpd_RM\ \ \text{and}\ \ \GFPid_RM:=\GFPCid_RM.$$
Recall that a left and right Noetherian ring $R$ is called {\it $n$-Gorenstein} if $\id_RR=\id_{R^{op}}R\leq n$.
The following result provides several equivalent characterizations for the category $\Mod R$ and the ring $R$ to be $n$-Gorenstein.

\begin{cor} \label{cor-4.7}
It holds that
\begin{enumerate}
\item[{\rm (i)}]
For any $n\ge 0$, the following statements are equivalent.
\begin{enumerate}
\item[{\rm (1.1)}] $\Mod R$ is $n$-Gorenstein $($that is, $\pd \mathcal{I}(R)\leq n$ and $\id\mathcal{P}(R)\leq n)$.
\item[{\rm (1.2)}] $\Gpd_RM\leq n$ for any $M\in\Mod R$.
\item[{\rm (1.3)}] $\Gpd\mathcal{I}(R)\leq n$ and $\id\mathcal{P}(R)\leq n$.
\item[{\rm (1.4)}] $\Gid_RN\leq n$ for any $N\in\Mod R$.
\item[{\rm (1.5)}] $\Gid\mathcal{P}(R)\leq n$ and $\pd \mathcal{I}(R)\leq n$.
\item[{\rm (1.6)}] $\id^{\leq n}(R)=\pd^{\leq n}(R)$.
\item[{\rm (1.7)}] $\id^{\leq n}(R)\subseteq\pd^{\leq n}(R)\subseteq\Gid^{\leq n}(R)$.
\item[{\rm (2.1)}] $\Mod R$ is right $n$-$(\mathcal{P}(R),\mathcal{F}(R))$-Gorenstein 
$($that is, $\pd \mathcal{I}(R)\leq n$ and $\id \mathcal{F}(R)\leq n)$.
\item[{\rm (2.2)}] $\SGFpd_RM\leq n$ for any $M\in\Mod R$.
\item[{\rm (2.3)}] $\SGFpd\mathcal{I}(R)\leq n$ and
$\id \mathcal{F}(R)\leq n$.
\item[{\rm (2.4)}] $\fd^{\leq n}(R)=\id^{\leq n}(R)=\pd^{\leq n}(R)$.
\item[{\rm (3.1)}] $\Mod R$ is left $n$-$(\mathcal{I}(R),\mathcal{FI}(R))$-Gorenstein 
$($that is, $\id \mathcal{P}(R)\leq n$ and $\pd \mathcal{FI}(R)\leq n)$.
\item[{\rm (3.2)}] $\GFPid_RM\leq n$ for any $M\in\Mod R$.
\item[{\rm (3.3)}] $\GFPid \mathcal{P}(R)\leq n$
and $\pd \mathcal{FI}(R)\leq n$.
\item[{\rm (3.4)}] $\FPid^{\leq n}(R)=\pd^{\leq n}(R)=\id^{\leq n}(R)$.
\item[{\rm (3.5)}] $\FPid^{\leq n}(R)\subseteq\pd^{\leq n}(R)\subseteq\GFPid^{\leq n}(R)$.
\end{enumerate}
\item[{\rm (ii)}]
If $R$ is a left Noetherian and right coherent ring, 
then all conditions in {\rm (i)} and the below condition are equivalent.
\begin{enumerate}
\item[{\rm (4)}] $\FPid_RR\leq n$ and $\FPid_{R^{op}}R\leq n$.
\end{enumerate}
\item[{\rm (iii)}]
If $R$ is a left and right Noetherian ring, 
then all conditions in {\rm (i)} and the below condition are equivalent.
\begin{enumerate}
\item[{\rm (5)}] $R$ is $n$-Gorenstein.
\end{enumerate}
\end{enumerate}
\end{cor}

\begin{proof}
By \cite[Theorem 5.6]{SS} and \cite[Theorem 2.6]{H} (respectively, \cite[Theorem 3]{GI} and \cite[Theorem 2.8]{YLL}), 
we have that the class of Gorenstein injective (respectively, Gorenstein FP-injective) left $R$-modules
is preenveloping and coresolving. On the other hand, if $R$ is a left and right Noetherian ring, then
$(1.2)\Longleftrightarrow (4)$ by \cite[Theorem 1.2]{Hu3}. Now, putting $_RC_S={_RR_R}$ in Theorem \ref{thm-4.6}, 
and then, by combining with Theorem \ref{thm-3.14}, we obtain the desired conclusion.
\end{proof}

As an immediate consequence of Theorem \ref{thm-4.6} and Corollary \ref{cor-4.7}, we get the following result, in which 
the first equality in the assertion (1) was proved in \cite[Corollary 5.13(1)]{HS}.

\begin{cor} \label{cor-4.8}
It holds that 
\begin{enumerate}
\item[{\rm (1)}]
\begin{align*}
&\sup\{\GCpd_RM\mid M\in\Mod R\}=\sup\{\GCid_SN\mid N\in\Mod S\} \\
=&\sup\{\SGFCpd_RM\mid M\in\Mod R\}=\sup\{\GFPCid_SN\mid N\in\Mod S\}.
\end{align*}
\item[{\rm (2)}] {\rm (\cite[Theorem 1.1]{BM}, \cite[Corollary 8.8]{HMP} and \cite[Theorem 1.1]{Wa})}
\begin{align*}
&\sup\{\Gpd_RM\mid M\in\Mod R\}=\sup\{\Gid_RM\mid M\in\Mod R\} \\
=&\sup\{\SGFpd_RM\mid M\in\Mod R\}=\sup\{\GFPid_RM\mid M\in\Mod R\}.
\end{align*}
\end{enumerate}
\end{cor}
 
By Theorems \ref{thm-3.11} and \ref{thm-3.14}, we also get the following result, 
which can be regarded as a supplement to Corollary \ref{cor-4.7}.

\begin{prop} \label{prop-4.9}
Let $n\geq 0$.
\begin{enumerate}
\item[{\rm (1)}] 
Let $\D$ be a subcategory of $\Mod R$ containing $\mathcal{P}(R)$. Then the following statements are equivalent.
\begin{enumerate}
\item[{\rm (1.1)}] $\Mod R$ is right $n$-$(\mathcal{P}(R),\D)$-Gorenstein 
$($that is, $\pd \mathcal{I}(R)\leq n$ and $\id \mathscr{D}\leq n)$.
\item[{\rm (1.2)}] $r\mathcal{G}(\mathcal{P}(R),\mathscr{D})$-$\pd M\leq n$ for any $M\in\Mod R$.
\item[{\rm (1.3)}] $r\mathcal{G}(\mathcal{P}(R),\mathscr{D})$-$\pd \mathcal{I}(R)\leq n$ and
$\id \mathscr{D}\leq n$.
\item[{\rm (1.4)}] $\mathscr{D}$-$\pd^{\leq n}(R)=\id^{\leq n}(R)=\pd^{\leq n}(R)$.
\end{enumerate}
\item[{\rm (2)}] 
Let $\D$ be a subcategory of $\Mod R$ containing $\mathcal{I}(R)$. Then the following statements are equivalent.
\begin{enumerate}
\item[{\rm (2.1)}] $\Mod R$ is left $n$-$(\mathcal{I}(R),\D)$-Gorenstein 
$($that is, $\id \mathcal{P}(R)\leq n$ and $\pd \mathscr{D}\leq n)$.
\item[{\rm (2.2)}] $l\mathcal{G}(\mathcal{I}(R),\mathscr{D})$-$\id M\leq n$ for any $M\in\Mod R$.
\item[{\rm (2.3)}] $l\mathcal{G}(\mathcal{I}(R),\mathscr{D})$-$\id \mathcal{P}(R)\leq n$
and $\pd \mathscr{D}\leq n$.
\item[{\rm (2.4)}] $\mathscr{D}$-$\id^{\leq n}=\pd^{\leq n}(R)=\id^{\leq n}(R)$.
\end{enumerate}
\end{enumerate}
\end{prop}

\begin{proof}
It is trivial that $\mathcal{P}(R)\subseteq r\mathcal{G}(\mathcal{P}(R),\mathscr{D})$ and 
$\mathcal{I}(R)\subseteq l\mathcal{G}(\mathcal{I}(R),\mathscr{D})$. 
Putting $\C=\mathcal{P}(R)$ in Theorem \ref{thm-3.11}, we get the assertion (1). 
Putting $\C=\mathcal{I}(R)$ in Theorem \ref{thm-3.14}, we get the assertion (2).
\end{proof}


Recall from \cite{BGH} that a module $M\in\Mod R$ is called {\it level}
if $M\in\mathcal{L}(R)$, where $\mathcal{L}(R)=\{M\in\Mod R\mid\Tor^R_1(X,M)=0$ 
(equivalently, $\Tor^R_{\geq 1}(X,M)=0$) for any right $R$-module $X$ 
admitting a degreewise finite $R^{op}$-projective resolution$\}$;
and a module $M\in\Mod R$ is called
{\it Gorenstein AC-projective} if $M\in\mathcal{GP}_{ac}(R)$, where
$$\mathcal{GP}_{ac}(R):=r\mathcal{G}(\mathcal{P}(R),\mathcal{L}(R)).$$
Also recall from \cite{BGH} that a module $M\in\Mod R$ is called {\it absolutely clean}
if $M\in\mathcal{AC}(R)$, where $\mathcal{AC}(R)=\{M\in\Mod R\mid\Ext_R^1(X,M)=0$ 
(equivalently, $\Ext_R^{\geq 1}(X,M)=0$) for any left $R$-module $X$ 
admitting a degreewise finite $R$-projective resolution$\}$;
and a module $M\in\Mod R$ is called
{\it Gorenstein AC-injective} if $M\in\mathcal{GI}_{ac}(R)$, where
$$\mathcal{GI}_{ac}(R):=l\mathcal{G}(\mathcal{I}(R),\mathcal{AC}(R)).$$

%
%

\begin{lem}\label{lem-4.10}
For any $M\in\Mod R$, it holds that
\begin{enumerate}
\item[{\rm (1)}] $\mathcal{AC}(R)\text{-}\id M=\mathcal{L}(R^{op})\text{-}\pd M^+$.
\item[{\rm (2)}] $\mathcal{L}(R)\text{-}\pd M=\mathcal{AC}(R^{op})\text{-}\id M^+$.
\end{enumerate}
\end{lem}

\begin{proof}
By \cite[Propositions 2.7(3) and 2.10(3)]{BGH}, we have that 
both $\mathcal{AC}(R)$ and $\mathcal{AC}(R^{op})$ are coresolving and closed under direct summands, and 
both $\mathcal{L}(R)$ and $\mathcal{L}(R^{op})$ are resolving and closed under direct summands.
On the other hand, by \cite[Theorem 2.12]{BGH}, we have 
\begin{enumerate}
\item[$(a)$] 
a module $M\in\Mod R$ (respectively, $\Mod R^{op}$), it holds that $M\in\mathcal{AC}(R)$ 
(respectively, $\mathcal{AC}(R^{op})$) 
if and only if $M^+\in\mathcal{L}(R^{op})$ (respectively, $\mathcal{L}(R)$);
\item[$(b)$]
a module $N\in\Mod R^{op}$ (respectively, $\Mod R$), it holds that $N\in\mathcal{L}(R^{op})$ 
(respectively, $\mathcal{L}(R)$)
if and only if $N^+\in\mathcal{AC}(R)$ (respectively, $\mathcal{AC}(R^{op})$).
\end{enumerate}
Now the assertion follows from Proposition \ref{prop-4.1}.
\end{proof}

The following result provides some special $n$-Gorenstein categories.

\begin{prop} \label{prop-4.11}
Let $n\geq 0$.
\begin{enumerate}
\item[{\rm (1)}] 
The following statements are equivalent.
\begin{enumerate}
\item[{\rm (1.1)}] $\Mod R$ is right $n$-$(\mathcal{P}(R),\mathcal{L}(R))$-Gorenstein 
$($that is, $\pd \mathcal{I}(R)\leq n$ and $\id \mathcal{L}(R)\leq n)$.
\item[{\rm (1.2)}] $\mathcal{GP}_{ac}(R)$-$\pd M\leq n$ for any $M\in\Mod R$.
\item[{\rm (1.3)}] $\mathcal{GP}_{ac}(R)$-$\pd \mathcal{I}(R)\leq n$ and
$\id \mathcal{L}(R)\leq n$.
\item[{\rm (1.4)}] $\mathcal{L}(R)$-$\pd^{\leq n}=\id^{\leq n}(R)=\pd^{\leq n}(R)$.
\item[{\rm (1.5)}] $\mathcal{L}(R)$-$\pd^{\leq n}\subseteq\id^{\leq n}(R)\subseteq 
r\mathcal{G}(\mathcal{P}(R),\mathcal{L}(R))$-$\pd^{\leq n}$.
\end{enumerate}
\item[{\rm (2)}] 
The following statements are equivalent.
\begin{enumerate}
\item[{\rm (2.1)}] $\Mod R$ is left $n$-$(\mathcal{I}(R),\mathcal{AC}(R))$-Gorenstein 
$($that is, $\id \mathcal{P}(R)\leq n$ and $\pd \mathcal{AC}(R)\leq n)$.
\item[{\rm (2.2)}] $\mathcal{GI}_{ac}(R)$-$\id M\leq n$ for any $M\in\Mod R$.
\item[{\rm (2.3)}] $\mathcal{GI}_{ac}(R)$-$\id \mathcal{P}(R)\leq n$
and $\pd \mathcal{AC}(R)\leq n$.
\item[{\rm (2.4)}] $\mathcal{AC}(R)$-$\id^{\leq n}=\pd^{\leq n}(R)=\id^{\leq n}(R)$.
\item[{\rm (2.5)}] $\mathcal{AC}(R)$-$\id^{\leq n}\subseteq\pd^{\leq n}(R)\subseteq 
l\mathcal{G}(\mathcal{I}(R),\mathcal{AC}(R))$-$\id^{\leq n}$.
\end{enumerate}
\end{enumerate}
If one of the equivalent conditions in $(1)$ or $(2)$ is satisfied, then it holds that
\begin{enumerate}
\item[{\rm (3)}] $\id_RR\leq n$ and $\mathcal{AC}(R^{op})\text{-}\id R\leq n$.
\item[{\rm (4)}] $\mathcal{AC}(R)\text{-}\id R\leq n$ and $\mathcal{AC}(R^{op})\text{-}\id R\leq n$.
\end{enumerate}
If $R$ is a left Noetherian and right coherent ring, then all the above conditions are equivalent.
\end{prop}

\begin{proof}
By \cite[Theorem 8.5 and Lemma 8.6]{BGH}, we have that $\mathcal{GP}_{ac}(R)$ is precovering and resolving.
By \cite[Theorem 5.5 and Lemma 5.6]{BGH}, we have that $\mathcal{GI}_{ac}(R)$ is preenveloping and coresolving.
Putting $\mathscr{C}=\mathcal{P}(R)$ and $\D=\mathcal{L}(R)$ in Theorem \ref{thm-3.11}, 
we get the assertion (1). Putting $\mathscr{C}=\mathcal{I}(R)$ and $\D=\mathcal{AC}(R)$ 
in Theorem \ref{thm-3.14}, we get the assertion (2). 


$(1.4)$ (respectively, $(2.4))\Longrightarrow (3)$ 
By (1.4) (respectively, (2.4)), we have $\id_RR\leq n$.
Since ${R}^+\in\mathcal{I}(R)$ by \cite[Theorem 2.1]{F}, we have
$$\mathcal{AC}(R^{op})\text{-}\id R=\mathcal{L}(R)\text{-}\pd{R}^+\leq\pd_R{R}^+\leq n$$
by the symmetric version of Lemma \ref{lem-4.10}(1) and (1.4) (respectively, (2.4)).

Since $\mathcal{I}(R)\subseteq\mathcal{AC}(R)$, we have $\mathcal{AC}(R)\text{-}\id R\leq\id_RR$, 
and thus $(3)\Longrightarrow (4)$ follows.

Assume that $R$ is a left Noetherian and right coherent ring. Then 
$$\mathcal{L}(R)=\mathcal{F}(R)\ \text{and}\ \mathcal{AC}(R)=\mathcal{FI}(R)=\mathcal{I}(R).$$
It yields $(4)\Longrightarrow (3)$ immediately. In addition, it yields 
$(1.4)\Longleftrightarrow (3)\Longleftrightarrow (2.4)$ by Corollary \ref{cor-4.7}.
%
\end{proof}

It is interesting to ask the following question.

\begin{ques}\label{ques-4.12}
{\rm In general, are all the conditions in Proposition \ref{prop-4.11}(1)(2) equivalent?}
\end{ques}

We claim that the answer to this question is positive provided that $R$ is either 
a left and right coherent ring or a commutative ring. The reasons are as follows: 
(i) If $R$ is a left and right coherent ring, then $\mathcal{L}(R)=\mathcal{F}(R)$ and 
$\mathcal{AC}(R)=\mathcal{FI}(R)$. So $\mathcal{L}(R)$-$\pd^{\leq n}=\fd^{\leq n}(R)$
and $\mathcal{AC}(R)$-$\id^{\leq n}=\FPid^{\leq n}(R)$, and hence the equivalence $(2.4)\Longleftrightarrow (3.4)$
in Corollary \ref{cor-4.7} implies the equivalence $(1.4)\Longleftrightarrow (2.4)$
in Proposition \ref{prop-4.11}. (ii) Let $R$ be a commutative ring. Suppose that
$\mathcal{L}(R)$-$\pd^{\leq n}=\id^{\leq n}(R)=\pd^{\leq n}(R)$ holds true. 
If $M\in\mathcal{AC}(R)$-$\id^{\leq n}$, then $M^+\in\mathcal{L}(R)$-$\pd^{\leq n}=\pd^{\leq n}(R)$
by Lemma \ref{lem-4.10}(1). It follows from \cite[Theorem 2.1]{F} that $M\in\id^{\leq n}(R)$. 
So $\mathcal{AC}(R)$-$\id^{\leq n}\subseteq\id^{\leq n}(R)$, and hence
$\mathcal{AC}(R)$-$\id^{\leq n}=\id^{\leq n}(R)=\pd^{\leq n}(R)$.
Conversely, suppose that $\mathcal{AC}(R)$-$\id^{\leq n}=\id^{\leq n}(R)=\pd^{\leq n}(R)$ holds true. 
If $M\in\mathcal{L}(R)$-$\pd^{\leq n}$, then $M^+\in\mathcal{AC}(R)$-$\id^{\leq n}=\pd^{\leq n}(R)$
by Lemma \ref{lem-4.10}(2). It again follows from \cite[Theorem 2.1]{F} that 
$M\in\id^{\leq n}(R)=\pd^{\leq n}(R)$. So $\mathcal{L}(R)$-$\pd^{\leq n}\subseteq\pd^{\leq n}(R)$, 
and hence $\mathcal{L}(R)$-$\id^{\leq n}=\id^{\leq n}(R)=\pd^{\leq n}(R)$. This proves 
the equivalence $(1.4)\Longleftrightarrow (2.4)$ in Proposition \ref{prop-4.11}.

 
The following example illustrates that the assumption ``$R$ is a left Noetherian and right coherent ring"
in Corollary \ref{cor-4.7}(ii) and the last assertion in Proposition \ref{prop-4.11} can not be weakened to that 
``$R$ is a left and right coherent ring", 
and hence the assumption ``$R$ is a left Noetherian ring and $S$ is a right coherent ring"
in Theorem \ref{thm-4.6}(iii) can not be weakened to that ``$R$ is a left coherent ring and $S$ is a right coherent ring".

\begin{exa} \label{exa-4.13}
{\rm If $R$ is a left and right coherent ring, then in Proposition \ref{prop-4.11}, the condition (3) is equivalent to
that $\id_RR\leq n$ and $\FPid_{R^{op}}R\leq n$, and the condition (4) is equivalent to that 
$\FPid_RR\leq n$ and $\FPid_{R^{op}}R\leq n$.

Let $D$ be a division ring and let $V$ be an infinite dimensional right $D$-vector space. Set $R:=\End(V_D)$.
Then $R$ is a von Neumann regular ring (that is, any left $R$-module is flat)
by \cite[Example 3.74A]{L}, and hence $R$ is a left and right coherent ring 
and $\FPid_RR=\FPid_{R^{op}}R=0$ by \cite[Example 4.46(b)]{L} and \cite[Theorem 5]{M}.  
Moreover, we have $\id_RR\neq 0$ and $\id_{R^{op}}R=0$ by \cite[Example 3.74B]{L}. 
In this case, we have $(4)\nRightarrow (3)$ in Proposition \ref{prop-4.11}.

Now set $R':=R^{op}$. Then $R'$ is also a von Neumann regular ring, and hence $R'$ is a left and right coherent ring
and $\FPid_{R'}R'=\FPid_{{R'}^{op}}R'=0$. By \cite[Theorem 5]{M}, we have $\Mod R'=\mathcal{FI}(R')$. 
In addition, we have $\id_{R'}R'=0$ and $\id_{{R'}^{op}}R'\neq 0$
by the above argument, which implies that $R'$ is not semisimple, and thus 
$\mathcal{P}(R')\subsetneqq \Mod R'=\mathcal{F}(R')$ and $\mathcal{I}(R')\subsetneqq \Mod R'=\mathcal{FI}(R')$.
In this case, we have $(3)\nRightarrow (1.4)$ and $(3)\nRightarrow (2.4)$ in Proposition \ref{prop-4.11}.}
\end{exa}

\subsection{Auslander and Bass classes}

%

We first establish the relation among the projective dimension of $C$ and certain relative homological dimensions.

\begin{lem} \label{lem-4.14}
It holds that
\begin{enumerate}
\item[{\rm (1)}] $\pd_RC=\pd\mathcal{P}_C(R)=\mathcal{P}_C(S^{op})$-$\id \mathcal{P}(S^{op})
=\mathcal{P}_C(S^{op})$-$\id S=\mathcal{F}_C(S^{op})$-$\id S=\mathcal{B}_C(S^{op})$-$\id S$

$\ \ \ \ \ \ \ \ =\id \mathcal{I}_C(R^{op})=\mathcal{I}_C(S)\text{-}\pd \mathcal{I}(S)
=\mathcal{I}_C(S)$-$\pd{S}^+=\mathcal{A}_C(S)$-$\pd{S}^+$.
\item[{\rm (2)}] $\pd_{S^{op}}C=\pd\mathcal{P}_C(S^{op})=\mathcal{P}_C(R)$-$\id \mathcal{P}(R)
=\mathcal{P}_C(R)$-$\id R=\mathcal{F}_C(R)$-$\id R=\mathcal{B}_C(R)$-$\id R$

$\ \ \ \ \ \ \ \ \ \ =\id \mathcal{I}_C(S)=\mathcal{I}_C(R^{op})\text{-}\pd \mathcal{I}(R^{op})
=\mathcal{I}_C(R^{op})$-$\pd{R}^+=\mathcal{A}_C(R^{op})$-$\pd{R}^+$.
\end{enumerate}
\end{lem}

\begin{proof}
(1) Since $\mathcal{P}_C(R)=\Add_RC$ by \cite[Proposition 2.4(a)]{LHX}, we have
$$\pd_RC=\pd \mathcal{P}_C(R).$$ By \cite[Lemma 4.3]{TH3}, we have
$$\pd_{R}C=\mathcal{P}_C(S^{op})\text{-}\id S.$$ Since
$\mathcal{P}_C(S^{op})\subseteq\mathcal{F}_C(S^{op})\subseteq\mathcal{B}_C(S^{op})$
by the symmetric version of \cite[Corollary 6.1]{HW}, we have
$$\mathcal{P}_C(S^{op})\text{-}\id S=\mathcal{F}_C(S^{op})\text{-}\id S=\mathcal{B}_C(S^{op})\text{-}\id S$$
by the symmetric version of \cite[Lemma 4.5]{TH3}.

Note that $\mathcal{P}_C(S^{op})=\Add C_S$ by the symmetric version of \cite[Proposition 2.4(a)]{LHX}.
If $\mathcal{P}_C(S^{op})\text{-}\id S=n<\infty$, then it is easy to see that
$\mathcal{P}_C(S^{op})\text{-}\id F\leq n$ for any free right $S$-module $F$.
It follows from the symmetric version of \cite[Lemma 4.6]{TH2}
that $\mathcal{P}_C(S^{op})$-$\id P\leq n$ for any projective right $S$-module $P$. This shows
$\mathcal{P}_C(S^{op})$-$\id \mathcal{P}(S^{op})\leq\mathcal{P}_C(S^{op})$-$\id S$, and hence
$$\mathcal{P}_C(S^{op})\text{-}\id \mathcal{P}(S^{op})=\mathcal{P}_C(S^{op})\text{-}\id S.$$

Since $\mathcal{I}_C(R^{op})=\Prod(_{R}C)^+$ 
by the symmetric version of \cite[Proposition 2.4(b)]{LHX}, we have $\id_{R^{op}}(_RC)^+=\id \mathcal{I}_C(R^{op})$.
On the other hand, by \cite[Theorem 2.1]{F}, we have
$\pd_{R}C=\fd_{R}C=\id_{R^{op}}(_RC)^+$, and hence
$$\pd_{R}C=\id \mathcal{I}_C(R^{op}).$$

Suppose $\mathcal{I}_C(S)$-$\pd {S}^{+}=n<\infty$. Since $\mathcal{I}_C(S)$ is closed under direct products by
\cite[Proposition 5.1(c)]{HW}, we have $\mathcal{I}_C(S)$-$\pd ({S}^{+})^J\leq n$ for any set $J$.
For any injective left $S$-module $E$, since $E$ is a direct summand of $({S}^{+})^J$ for some set $J$,
we have $\mathcal{I}_C(S)$-$\pd E\leq n$ by \cite[Lemma 4.7]{TH3}. This shows
$\mathcal{I}_C(S)$-$\pd \mathcal{I}(S)\leq\mathcal{I}_C(S)$-$\pd {S}^{+}$, and hence
$$\mathcal{I}_CS)\text{-}\pd \mathcal{I}(S)=\mathcal{I}_C(S)\text{-}\pd {S}^{+}.$$

By \cite[Lemma 4.8]{TH3} and the symmetric versions of \cite[Proposition 3.2]{Hu4} and Proposition \ref{prop-4.1}(1), we have
$$\mathcal{I}_C(S)\text{-}\pd{S}^+=\mathcal{A}_C(S)\text{-}\pd{S}^+=\mathcal{B}_C(S^{op})\text{-}\id S.$$

(2) It is the symmetric version of (1).
\end{proof}


By \cite[Theorems 1 and 6.1]{HW} (cf. \cite[Theorem 3.11(1)]{TH3} and \cite[Theorem 3.9]{TH1}), we have
$$\mathcal{A}_C(S)={^{\bot}\mathcal{I}_C(S)}\cap\widetilde{\cores\mathcal{I}_C(S)}=r\mathcal{G}(\mathcal{I}_C(S)),$$
$$\mathcal{A}_C(R^{op})={^{\bot}\mathcal{I}_C(R^{op})}\cap\widetilde{\cores\mathcal{I}_C(R^{op})}=r\mathcal{G}(\mathcal{I}_C(R^{op})),$$
$$\mathcal{B}_C(R)={\mathcal{P}_C(R)^{\bot}}\cap\widetilde{\res\mathcal{P}_C(R)}=l\mathcal{G}(\mathcal{P}_C(R)),$$
$$\mathcal{B}_C(S^{op})={\mathcal{P}_C(S^{op})^{\bot}}\cap\widetilde{\res\mathcal{P}_C(S^{op})}=l\mathcal{G}(\mathcal{P}_C(S^{op})).$$
The following result greatly improves \cite[Theorems 4.2 and 4.10]{TH3}. 

\begin{thm}\label{thm-4.15}
For any $n\ge 0$, the following statements are equivalent.
\begin{enumerate}
\item[{\rm (1)}] $\pd_RC=\pd_{S^{op}}C\leq n$.
\item[{\rm (a-1)}] $\Mod S$ is right $n$-$\mathcal{I}_C(S)$-Gorenstein 
$($that is, $\mathcal{I}_C(S)$-$\pd \mathcal{I}(S)\leq n$ and $\id \mathcal{I}_C(S)\leq n)$.
\item[{\rm (a-2)}] $\mathcal{A}_C(S)$-$\pd N\leq n$ for any $N\in\Mod S$.
\item[{\rm (a-3)}] $\mathcal{A}_C(S)$-$\pd \mathcal{I}(S)\leq n$ and $\id \mathcal{I}_C(S)\leq n$.
\item[{\rm (a-4)}] $\mathcal{A}_C(S)$-$\pd {S}^{+}\leq n$ and $\id \mathcal{I}_C(S)\leq n$.
\item[{\rm (a-5)}] $\mathcal{I}_C(S)$-$\pd {S}^{+}\leq n$ and $\id \mathcal{I}_C(S)\leq n$.
\item[{\rm (a-6)}] $\mathcal{I}_C(S)$-$\pd^{\leq n}\subseteq\id^{\leq n}(S)\subseteq\mathcal{A}_C(S)$-$\pd^{\leq n}$.
\item[{\rm (a-}$i)^{op}$] The $R^{op}$-version of {\rm (a-}$i)$ with $1\leq i\leq 6$.
\item[{\rm (b-1)}] $\Mod R$ is left $n$-$\mathcal{P}_C(R)$-Gorenstein 
$($that is, $\mathcal{P}_C(R)$-$\id \mathcal{P}(R)\leq n$ and $\pd \mathcal{P}_C(R)\leq n)$.
\item[{\rm (b-2)}] $\mathcal{B}_C(R)$-$\id M\leq n$ for any $M\in\Mod R$.
\item[{\rm (b-3)}] $\mathcal{B}_C(R)$-$\id \mathcal{P}(R)\leq n$ and $\pd \mathcal{P}_C(R)\leq n$.
\item[{\rm (b-4)}] $\mathcal{B}_C(R)$-$\id R\leq n$ and $\pd \mathcal{P}_C(R)\leq n$.
\item[{\rm (b-5)}] $\mathcal{P}_C(R)$-$\id R\leq n$ and $\pd \mathcal{P}_C(R)\leq n$.
\item[{\rm (b-6)}] $\mathcal{P}_C(R)$-$\id^{\leq n}\subseteq\pd^{\leq n}(R)\subseteq\mathcal{B}_C(R)$-$\id^{\leq n}$.
\item[{\rm (b-}$i)^{op}$] The $S^{op}$-version of {\rm (b-}$i)$ with $1\leq i\leq 6$.
\end{enumerate}
\end{thm}

\begin{proof}
By \cite[Theorem 3.3]{Hu4} and \cite[Theorem 6.2]{HW}, we have that both $\mathcal{A}_C(S)$ and $\mathcal{A}_C(R^{op})$
are precovering and resolving, and that both $\mathcal{B}_C(R)$ and $\mathcal{B}_C(S^{op})$ are preenveloping and coresolving.
Putting $\C=\D=\mathcal{I}_C(S)$ in Theorem \ref{thm-3.11}, we get 
(a-1)$\Longleftrightarrow$(a-2)$\Longleftrightarrow$(a-3)$\Longleftrightarrow$(a-6).
Putting $\C=\D=\mathcal{P}_C(R)$ in Theorem \ref{thm-3.14}, we get
(b-1)$\Longleftrightarrow$(b-2)$\Longleftrightarrow$(b-3)$\Longleftrightarrow$(b-6).
By Lemma \ref{lem-4.14}, we have
(1)$\Longleftrightarrow$(a-3)$\Longleftrightarrow$(a-4)$\Longleftrightarrow$(a-5) and
(1)$\Longleftrightarrow$(b-3)$\Longleftrightarrow$(b-4)$\Longleftrightarrow$(b-5), and
(a-$4)^{op}\Longleftrightarrow$(b-$4)$.
By symmetry, the proof is finished.
\end{proof}

As an immediate consequence, we obtain the following result.

\begin{cor}\label{cor-4.16}
It holds that
\begin{align*}
& \sup\{\mathcal{A}_C(S)\text{-}\pd N\mid N\in\Mod S\}
=\sup\{\mathcal{A}_C(R^{op})\text{-}\pd N'\mid N'\in\Mod R^{op}\}\\
=&\sup\{\mathcal{B}_C(R)\text{-}\id M\mid M\in\Mod R\}
=\sup\{\mathcal{B}_C(S^{op})\text{-}\id M'\mid M'\in\Mod S^{op}\}.
\end{align*}
\end{cor}

We give the following remark.

\begin{rem}\label{rem-4.17}
{\rm Compare Theorems \ref{thm-4.6}(i) and \ref{thm-4.15}.
\begin{enumerate}
\item[{\rm (1)}]
It is natural to ask whether the conditions in Theorem \ref{thm-4.6}(i) are left-right symmetric.
In general, the answer is negative. Otherwise, if the conditions in Theorem \ref{thm-4.6}(i) are left-right symmetric,
then so are those in Corollary \ref{cor-4.7}(i). It yields
$$\{\Gpd_RM\mid M\in\Mod R\}=\{\Gpd_{R^{op}}N\mid N\in\Mod R^{op}\}.$$
By \cite[Proposition 2.27]{H}, if $\gldim R<\infty$ and $\gldim R^{op}<\infty$, where 
$\gldim R$ and $\gldim R^{op}$ are the left and right global dimensions of $R$, respectively,
then $$\gldim R=\{\Gpd_RM\mid M\in\Mod R\}\ \text{and}\ \gldim R^{op}=\{\Gpd_{R^{op}}N\mid N\in\Mod R^{op}\}.$$
On the other hand, for any $0\leq m\neq n<\infty$, there exists a ring $R$ such that 
$$\gldim R=m\ \text{and}\ \gldim R^{op}=n$$
by \cite[p.439, Corollary]{J}. In this case, we have 
$$\{\Gpd_RM\mid M\in\Mod R\}=m\neq n=\{\Gpd_{R^{op}}N\mid N\in\Mod R^{op}\},$$
which is a contradiction.
\item[{\rm (2)}]
For any $n\geq 0$, it is also natural to consider the following conditions:
\begin{enumerate}
\item[{\rm (i)}] $\Mod R$ is left $n$-$\mathcal{P}_C(R)$-Gorenstein 
$($that is, $\mathcal{P}_C(R)$-$\id \mathcal{P}(R)\leq n$ and $\pd \mathcal{P}_C(R)\leq n)$.
\item[{\rm (ii)}] $\Mod S$ is right $n$-$\mathcal{I}_C(S)$-Gorenstein 
$($that is, $\mathcal{I}_C(S)$-$\pd \mathcal{I}(S)\leq n$ and $\id \mathcal{I}_C(S)\leq n)$.
\item[{\rm (iii)}] $\Mod R$ is left $n$-($\mathcal{P}_C(R),\mathcal{F}_C(R))$-Gorenstein 
$($that is, $\mathcal{P}_C(R)$-$\id \mathcal{P}(R)\leq n$ and $\pd \mathcal{F}_C(R)$ $\leq n)$.
\item[{\rm (iv)}] $\Mod S$ is right $n$-($\mathcal{I}_C(S),\mathcal{FI}_C(S))$-Gorenstein 
$($that is, $\mathcal{I}_C(S)$-$\pd \mathcal{I}(S)\leq n$ and $\id \mathcal{FI}_C(S)$ $\leq n)$.
\end{enumerate}
By Theorem \ref{thm-4.15}, we have (i)$\Longleftrightarrow$(ii). It is trivial that 
(iii)$\Longrightarrow$(i) and (iv)$\Longrightarrow$(ii). Putting $_RC_S={_RR_R}$ and $n=0$, we have that 
the assertions (i) and (ii) hold true. In this case, the assertion (iii) holds true 
$\Longleftrightarrow$ $\pd\mathcal{F}(R)=0$ $\Longleftrightarrow$ $R$ is a left perfect ring 
by \cite[Theorem 5.3.2]{EJ}; and the assertion (iv) holds true $\Longleftrightarrow$ 
$\id\mathcal{FI}(R)=0$ $\Longleftrightarrow$ $R$ is a left Noetherian ring by \cite[Theorem 3]{M}. 
Thus, in general, we have
$${\rm (i)}+{\rm (ii)}\nRightarrow{\rm (iii)},\  {\rm (i)}+{\rm (ii)}\nRightarrow{\rm (iv)},\ 
{\rm (iii)}\nRightarrow{\rm (iv)}\ \ \text{and}\ \ {\rm (iv)}\nRightarrow{\rm (iii)}.$$
\end{enumerate}}
\end{rem}

In the following, we apply Theorem \ref{thm-4.15} to coherent semilocal rings.
We need the following result.

\begin{prop}\label{prop-4.18}
It holds that
\begin{enumerate}
\item[{\rm (1)}] Let $I$ be an ideal of $R$ such that $R/I$ is a semisimple ring. Then
$$\mathcal{B}_C(R)\text{-}\id R/I=\mathcal{A}_C(R^{op})\text{-}\pd R/I$$
provided that both of them are finite.
\item[{\rm (2)}] Let $I'$ be an ideal of $S$ such that $S/I'$ is a semisimple ring. Then
$$\mathcal{B}_C(S^{op})\text{-}\id S/I'=\mathcal{A}_C(S)\text{-}\pd S/I'$$
provided that both of them are finite.
\end{enumerate}
\end{prop}

\begin{proof}
(1) For any $n\geq 0$, we have
\begin{align*}
& \mathcal{B}_C(R)\text{-}\id R/I\leq n\\
\Longleftrightarrow & \Ext_R^{\geq n+1}(C,R/I)=0\ \ \ \text{(by \cite[Theorem 4.2]{TH2})}\\
\Longleftrightarrow & \Tor^R_{\geq n+1}(R/I,C)=0\ \ \ \ \text{(by \cite[Lemma 3.1]{SH})}\\
\Longleftrightarrow & \mathcal{A}_C(R^{op})\text{-}\pd R/I\leq n.\ \ \ \text{(by \cite[Theorem 4.4]{Hu4})}
\end{align*}

(2) It is the symmetric version of (1).
\end{proof}

Recall that a ring $R$ is called {\it semilocal} if $R/J(R)$ is a semisimple ring,
where $J(R)$ is the Jacobson radical of $R$.

\begin{prop}\label{prop-4.19}
It holds that
\begin{enumerate}
\item[{\rm (1)}] Let $R$ be a left coherent semilocal ring. Then
$$\pd_RC\leq\min\{\mathcal{B}_C(R)\text{-}\id R/J(R),\mathcal{A}_C(R^{op})\text{-}\pd R/J(R)\}.$$
Furthermore, if $\pd_{S^{op}}C<\infty$, then 
$$\pd_RC=\mathcal{B}_C(R)\text{-}\id R/J(R)=\mathcal{A}_C(R^{op})\text{-}\pd R/J(R).$$
\item[{\rm (2)}]
Let $S$ be a right coherent semilocal ring. Then
$$\pd_{S^{op}}C\leq\min\{\mathcal{B}_C(S^{op})\text{-}\id S/J(S),\mathcal{A}_C(S)\text{-}\pd S/J(S)\}.$$
Furthermore, if $\pd_RC<\infty$, then 
$$\pd_{S^{op}}C=\mathcal{B}_C(S^{op})\text{-}\id S/J(S)=\mathcal{A}_C(S)\text{-}\pd S/J(S).$$
\end{enumerate}
\end{prop}

\begin{proof}
(1) 
If $\mathcal{B}_C(R)\text{-}\id R/J(R)=n<\infty$, then $\Ext_R^{\geq n+1}(C,R/J(R))=0$
by \cite[Theorem 4.2]{TH2}, and hence $\pd_RC\leq n$ by \cite[Theorem 2(1)]{XC}. This proves
$\pd_RC\leq\mathcal{B}_C(R)\text{-}\id R/J(R)$.

If $\mathcal{A}_C(R^{op})\text{-}\pd R/J(R)=n<\infty$, then 
$\Tor^R_{\geq n+1}(R/J(R),C)=0$ by \cite[Theorem 4.4]{Hu4}. It follows from \cite[Lemma 3.1]{SH} that
$\Ext_R^{\geq n+1}(C,R/J(R))=0$, and thus $\pd_RC\leq n$ by \cite[Theorem 2(1)]{XC} again. 
This proves $\pd_RC\leq\mathcal{A}_C(R^{op})\text{-}\pd R/J(R)$. 

Now suppose $\pd_{S^{op}}C<\infty$. Let $\pd_RC=n<\infty$. Then $\pd_RC=\pd_{S^{op}}C=n$
by \cite[Proposition 4.1]{TH3}. Thus $\mathcal{B}_C(R)\text{-}\id R/J(R)\leq n$ and 
$\mathcal{A}_C(R^{op})\text{-}\pd R/J(R)\leq n$ by Theorem \ref{thm-4.15}. It follows from 
Proposition \ref{prop-4.18}(1) that
$$\mathcal{A}_C(R^{op})\text{-}\pd R/J(R)=\mathcal{B}_C(R)\text{-}\id R/J(R)\leq n=\pd_RC.$$
The proof is finished.

(2) It is the symmetric version of (1).
\end{proof}

The Wakamatsu tilting conjecture states that if $R$ and $S$ are artin algebras, then $\pd_RC=\pd_{S^{op}}C$
(\cite{BR,MR}). The following result provides a necessary condition for the validity of this conjecture.

\begin{cor}\label{cor-4.20}
Let $R$ be a left coherent semilocal ring and $S$ a right coherent semilocal ring.
If $\pd_RC=\pd_{S^{op}}C$, then
$$\mathcal{B}_C(R)\text{-}\id R/J(R)=\mathcal{B}_C(S^{op})\text{-}\id S/J(S)=
\mathcal{A}_C(R^{op})\text{-}\pd R/J(R)=\mathcal{A}_C(S)\text{-}\pd S/J(S).\eqno{(4.6)}$$
\end{cor}

\begin{proof}
Suppose $\mathcal{B}_C(R)\text{-}\id R/J(R)=n<\infty$. Then
$\pd_RC=\pd_{S^{op}}C=n$ by assumption and Proposition \ref{prop-4.19}(1),
and thus
$$\mathcal{B}_C(S^{op})\text{-}\id S/J(S)=
\mathcal{A}_C(R^{op})\text{-}\pd R/J(R)=\mathcal{A}_C(S)\text{-}\pd S/J(S)=n$$
by Proposition \ref{prop-4.19}. This shows that the the first quantity in (4.6)
is at least the last three ones. Similarly, we get that any quantity in (4.6)
is at least the other three ones.
\end{proof}

As a consequence, we get the following result.

\begin{cor}\label{cor-4.21}
Let $R$ be a left coherent semilocal ring and $S$ a right coherent semilocal ring.
Then the following statements are equivalent.
\begin{enumerate}
\item[{\rm (1)}] $\pd_RC=\pd_{S^{op}}C<\infty$.
\item[{\rm (2)}] $\mathcal{B}_C(R)\text{-}\id R/J(R)<\infty$ and $\mathcal{B}_C(S^{op})\text{-}\id S/J(S)<\infty$.
\item[{\rm (3)}] $\mathcal{A}_C(R^{op})\text{-}\pd R/J(R)<\infty$ and $\mathcal{A}_C(S)\text{-}\pd S/J(S)<\infty$.
\end{enumerate}
If one of the above conditions is satisfied, then all these six quantities are identical.
\end{cor}

\begin{proof}
The implications that $(1)\Longrightarrow (2)$ and $(1)\Longrightarrow (3)$ follow from Theorem \ref{thm-4.15}, and 
the implications that $(2)\Longrightarrow (1)$ and $(3)\Longrightarrow (1)$ follow from Proposition \ref{prop-4.19}.
The last assertion follows from Corollary \ref{cor-4.20}.
\end{proof}

\section*{Acknowledgements} 
The author thanks Jiangsheng Hu and Junpeng Wang for constructive communication, and thanks the referee for useful suggestions.

\section*{Data availability} 
No data was used for the research described in the article.

\section*{Declarations} 
{\bf Conflict of interest} The author declares that he has no conflict of interest.

\end{document}